\documentclass[11pt]{article}
\usepackage[T1]{fontenc}
\usepackage{hyperref}
\usepackage{amsmath}
\usepackage{amsthm}
\usepackage{amsfonts}
\usepackage{mathtools}
\usepackage{amssymb}
\usepackage[all]{xy}

\newtheorem{thm}{Theorem}[section]

\newtheorem{lem}[thm]{Lemma}

\theoremstyle{definition}

\theoremstyle{remark}

\newcommand{\R}{\mathbb{R}}
\newcommand{\C}{\mathbb{C}}
\newcommand{\Z}{\mathbb{Z}}
\newcommand{\N}{\mathbb{N}}

\newcommand{\Do}{\mathcal{D}}

\DeclareMathOperator*{\im}{Im}%
\DeclareMathOperator*{\Span}{span}

\begin{document}


\title{Toeplitz Operators on Two Poly-Bergman-Type Spaces of the 
	Siegel Domain $D_2 \subset \mathbb{C}^2$ with Continuous Nilpotent Symbols.}

\author
{Yessica Hern\'andez-Eliseo \\
	Centro de Investigaci\'on en Matem\'aticas, M\'exico \\
E-mail: yeherel@gmail.com \\ \\
%
	Josu\'e Ram\'irez-Ortega \\ 
	Universidad Veracruzana, M\'exico \\
E-mail: josramirez@uv.mx \\ \\
%
	Francisco G. Hern\'andez-Zamora \\ 
	Universidad Veracruzana, M\'exico \\
E-mail: francischernandez@uv.mx
}

%

\date{Mayo, 2023}

	\maketitle
	\begin{abstract} 
We describe certain $C^*$-algebras generated by Toeplitz operators with nilpotent symbols and acting on a poly-Bergman type space of the Siegel domain $D_{2} \subset \mathbb{C}^{2}$.  Bounded measurable functions of  the form $\tilde{c}(\zeta) = c(\text{Im}\, \zeta_{1}, \text{Im}\, \zeta_{2} - |\zeta_1|^{2})$ are called nilpotent symbols. In this work we consider symbols of the form $\tilde{a}(\zeta) = a(\text{Im}\, \zeta_1)$ and $\tilde{b}(\zeta) = b(\text{Im}\, \zeta_2 -|\zeta_1|^{2})$, where both limits  $\lim\limits_{s\rightarrow 0^+} b(s)$ and  $\lim\limits_{s\rightarrow +\infty} b(s)$ exist, and $a$ belongs to the set of piece-wise continuous functions on $\overline{\mathbb{R}}=[-\infty,+\infty]$ and with one-sided limits at $0$. We describe certain $C^*$-algebras generated by such Toeplitz operators that turn out to be isomorphic to subalgebras of $M_n(\mathbb{C}) \otimes C(\overline{\Pi})$, where $\overline{\Pi}=\overline{\R} \times \overline{\R}_+$ and $\overline{\mathbb{R}}_+=[0,+\infty]$.
	\end{abstract}


	\section{Introduction.}

In recent years, the theory of Toeplitz operators has been generalized from Bergman spaces of square-integrable holomorphic functions to poly-Bergman spaces of square-integrable polyanalytic functions \cite{vasilev-SPB, vasilev-SPF}. Bianalytic functions em\-erged in the mathematical theory of elasticity, but the mathematical relevance of more general polyanalytic functions was soon realized \cite{Balk}.

Similar to the study of Toeplitz operators on spaces of analytic functions, we select a set of symbols  $E \subset L^{\infty}$ in such a way that the $C^*$-algebra and von Neumann algebra generated by Toeplitz operators with symbols in $ E $ can be explicitly described up to isomorphism, say, as an algebra of matrix-valued functions.

In this paper, we study Toeplitz operators with nilpotent symbols and acting on 
a poly-Bergman-type space on the Siegel domain $D_{2} \subset \mathbb{C}^{2}$. 
In \cite{QV-Ball1,QV-Ball2,LibroVasilevski},  the authors fully described all commutative
$C^*$-algebras generated by Toeplitz operators acting on the Bergman spaces of both the unit disk $\mathbb{D}$ and the Siegel domain  $D_n \subset \C^n$ with symbols invariant under the action
of a maximal Abelian subgroup of biholomorphisms.

Let $\Pi=\{z=x+iy \in \C \mid  y>0\}$ be the upper half-plane. 
Toeplitz operators with vertical symbols (which depend  on $y=\im z$) acting on  poly-Bergman-type  spaces are well studied.
Taking vertical symbols with limits at $y=0$ and $y=\infty$, in \cite{R-SN,R-RM-V} the authors described the $C^*$-algebra generated by all Toeplitz operators on the poly-Bergman space $\mathcal{A}^{2}_{n}(\Pi)$;
this algebra is isomorphic to a subalgebra of $M_n(\mathbb{C}) \otimes C[0,+\infty]$.
Similar research has studied Toeplitz operators on poly-Bergman spaces with homogeneous symbols (\cite{R-LL,R-RM-H}). In \cite{Hutnik,Hutnik-2,R-RM-V},  the authors  studied 
Toeplitz operators acting on the true-poly-Bergman space $\mathcal{A}^{2}_{(n)}(\Pi)$ from the point of view of wavelet spaces.
 Taking horizontal symbols having one-sided limits at $x=\pm \infty$,	in \cite{A-C-R-N, Esmeral-Vas} the authors studied Toeplitz operators acting on poly-Fock spaces $F_{k}^{2}(\C)$, they proved that the $C^*$-algebra generated by such Toeplitz operators is isomorphic to a subalgebra of
 $M_n(\mathbb{C}) \otimes C[-\infty,+\infty]$. In \cite{ME-TR, BMR, Bauer-Herr-Vas, Grudsky}, the authors  studied  the decomposition of the von Neumann algebra of radial operators acting on the poly-Fock spaces $F_{k}^{2}(\C)$ and weighted  poly-Bergman spaces $\mathcal{A}^{2}_{n}(\mathbb{D})$.

In \cite{QV-Ball1, QV-Ball2} the authors made remarkable progress in the study of Toeplitz operators acting
on the Bergman space of the Siegel domain $D_{n} \subset \mathbb{C}^n$. In particular, they studied
the $C^*$-algebra ${\cal T}_{{\cal N}_n}$   generated by all Toeplitz operators with
bounded nilpotent symbols, which are functions of the form  
$\tilde{c}(\zeta)=c(\im \zeta_{1},...,\im \zeta_{n-1}, \im \zeta_{n} - |\zeta'|^{2})$,
where $\zeta'=(\zeta_1,...,\zeta_{n-1})$.

Furthermore, in \cite{Y-R} the authors studied Toeplitz operators on the true-poly-Bergman-type space $\mathcal{A}^{2}_{(L)}(D_{2})$, with nilpotent symbols of the form
$\tilde{a}(\zeta) = a(\im \zeta_1)$ and $\tilde{b}(z) = b(\im \zeta_2 -|\zeta_1|^{2})$. 
The main purpose of this work is to describe the $C^*$-algebra generated by all the Toeplitz operators acting on the poly-Bergman-type spaces $\mathcal{A}^{2}_{(1,n)}(D_{2})$ and $\mathcal{A}^{2}_{(n,1)}(D_{2})$, we take nilpotent symbols of the form  $\tilde{a}(\zeta) = a(\im \zeta_1)$ and $\tilde{b}(\zeta) = b(\im \zeta_2 -|\zeta_1|^{2})$.

This paper is organized as follows.
In Section \ref{identificacion} we recall how poly-Bergman-type spaces are defined for the Siegel domain, and how
they can be identified with an $L^2$-space through a Bargmann-type transform.
In Section \ref{OT} we introduce the Toeplitz operators acting on $\mathcal{A}^{2}_{(1,n)}(D_{2})$  with  nilpotent symbols, and
we show that such Toeplitz operators are unitarily equivalent to multiplication operators. 

In Section \ref{seccion-sb} we take symbols of the form
$\tilde{b}(\zeta) = b(\im \zeta_2 -|\zeta_1|^{2})$  for which both limits $\lim\limits_{s\rightarrow 0^+} b(s)$
and  $\lim\limits_{s\rightarrow +\infty} b(s)$ exist;
the $C^*$-algebra generated by all Toeplitz operators $T_b$ is 
isomorphic to $\mathfrak{C}= \{M \in  M_n(\C)\otimes C[0, \infty]: M(0), M(+\infty) \in \C I\}$.
Let $PC(\overline{\R}, \{0\})$ denote  the set of all functions continuous on $\overline{\R}\setminus \{0\}$  and having one-side limit values at $0$, where $\overline{\R}$ is the two-point compactification
of $\R$.  In Section \ref{Alg-a} we take nilpotent symbols of the form
$\tilde{a}(\zeta) = a(\im \zeta_1)$, where $a\in PC(\overline{\R}, \{0\})$;
the $C^*$-algebra generated by all Toeplitz
operators $T_a$ is isomorphic to $C(\overline{\Pi})$, where  $\overline{\Pi}=\overline{\R} \times \overline{\R}_+$. 

In Section \ref{symb-discont} we introduce Toeplitz operators acting on $\mathcal{A}^{2}_{(n,1)}(D_{2})$ with  nilpotent symbols $\tilde{c}$, we show that such Toeplitz operators are unitarily equivalent to multiplication operators $\gamma^{c}I$ acting on $L^2 (\R \times \R_+)$, where $\gamma^{c}$ is a continuous matrix-valued function on $\Pi$. In this work we take symbols of the form $\tilde{a}(\zeta) = a(\im \zeta_1)$, where  $a \in PC(\overline{\R}, \{0\})$. In Section \ref{con-func} we prove that the matrix-valued function $\gamma^{a}$ can be continuously extended to $\overline{\Pi}$ under a change of variable, this is one of our main results.
In Section \ref{sep-est}  we prove that the $C^*$-algebra generated by all Toeplitz
operators $T_a$ is isomorphic to a $C^*$-subalgebra of
$M_n(\C)\otimes C(\overline{\Pi})$, thus the spectrum of such algebra is fully described.

	
	\section{Poly-Bergman type spaces of the Siegel domain.}\label{identificacion}
	
In this section, we recall  some results obtained in  \cite{R-SN-D} that are needed in this paper.
	Each $\zeta\in \mathbb{C}^n$ will be represented as an ordered pair
	$\zeta=(\zeta',\zeta_n)$, where $\zeta'=(\zeta_1,...,\zeta_{n-1}) \in \mathbb{C}^{n-1}$. The Euclidean norm will
	be denoted by $|\cdot|$. The $n$-dimensional Siegel domain is defined by
	$$D_{n}=\{\zeta=(\zeta',\zeta_{n}) \in \C^{n-1} \times \C : \im \zeta_{n}- |\zeta'|^{2} >0\}.$$
	We will study Toeplitz operators acting on certain poly-Bergman-type subspaces of 
  $L^{2}(D_{n}, d\mu_{\lambda})$,
	 where
	$d\mu_{\lambda}(\zeta)=(\im \zeta_{n}-|\zeta'|^{2})^{\lambda} d\mu (\zeta)$,  with $ \lambda >-1$,
	and $d\mu(\zeta)$ is the usual Lebesgue measure. 
	We clarify that if $X$ is any positive-measure subset of a Euclidean space, then $L^2(X)$ refers to 
	$L^2(X,d\mu)$, where $d\mu$ is Lebesgue measure restricted to $X$.
	
	For each multi-index $L=(l_{1},...,l_{n}) \in \mathbb{N}^n$, 
	the poly-Bergman-type space $\mathcal{A}^{2}_{\lambda L}(D_{n})$ is the closed subspace of $L^{2}(D_{n}, d\mu_{\lambda})$ 
	consisting of all $L$-analytic functions, that is, all functions $f\in C^\infty(D_n)$ satisfying the equations
	\begin{eqnarray*}
		\left( \frac{\partial}{\partial \overline{\zeta}_{m}}-2i\zeta_{m}\frac{\partial}{\partial \overline{\zeta}_{n}}  \right)^{l_{m}} f 
		&=&0, \;\; 1\leq m \leq n-1 \\
		\left( \frac{\partial}{\partial \overline{\zeta}_{n}} \right)^{l_{n}}f 
		&=&0.
	\end{eqnarray*}
	
	In particular, for  $L=(1,...,1)$, $\mathcal{A}^{2}_{\lambda L}(D_{n})$ is just the Bergman space.
	Likewise, the anti-poly-Bergman type  space $\widetilde{\mathcal{A}}^2_{\lambda L}(D_{n})$ is defined to be the complex conjugate
	of $\mathcal{A}^{2}_{\lambda L}(D_{n}).$ 
	Thus, we introduce  true-poly-Bergman-type spaces as follows:
	\begin{eqnarray*}
		\mathcal{A}^{2}_{\lambda(L)}(D_{n}) &=& \mathcal{A}^{2}_{\lambda L}(D_{n}) \ominus 
		\left( \sum_{m=1}^{n} \mathcal{A}^{2}_{\lambda,L-e_{m}}(D_{n}) \right),\\
		\widetilde{\mathcal{A}}^{2}_{\lambda(L)}(D_{n}) &=& \widetilde{\mathcal{A}}^{2}_{\lambda L}(D_{n}) \ominus 
		\left( \sum_{m=1}^{n} \widetilde{\mathcal{A}}^{2}_{\lambda,L-e_{m}}(D_{n}) \right),
	\end{eqnarray*}
	where  $e_{m}=(0,...,1,...,0)$ and the $1$ is placed at the $m$ entry. We assume that 
	$\mathcal{A}^{2}_{\lambda L}(D_{n})=\{0 \}$ whenever $L\in \Z^n \setminus \N^n$.
	
	In \cite{R-SN-D} the authors proved that $L^{2}(D_{n},d\mu_{\lambda})$ 
	is the direct sum of all the true-poly-Bergman-type spaces:
	$$L^{2}(D_{n}, d\mu_{\lambda})= 
	\left( \bigoplus_{L\in \mathbb{N}^{n} } \mathcal{A}^{2}_{\lambda(L)}(D_{n}) \right) \bigoplus 
	\left( \bigoplus_{L\in \mathbb{N}^{n} } \widetilde{\mathcal{A}}^{2}_{\lambda(L)}(D_{n}) \right).$$
	The authors also constructed a unitary map from $\mathcal{A}^{2}_{\lambda(L)}(D_{n})$  to the tensor product
	$$L^{2}(\R^{n-1})\otimes {\cal H}_{l_{1}-1}\otimes \cdot\cdot\cdot\otimes {\cal H}_{l_{n-1}-1} \otimes 
	L^{2}(\R_{+}) \otimes {\cal L}_{l_{n}-1},$$ 
	where $\mathbb{R}_+=(0,\infty)$.  Both the ${\cal H}_{m}$ and ${\cal L}_{m}$ are one-dimensional spaces defined below. 
	Recall the Hermite and Laguerre polynomials:
	$$H_m(y):= (-1)^m e^{y^2} \frac{d^m}{dy^m} (e^{-y^2}), \quad 
	L_m^{\lambda}(y):= e^y \frac{y^{-\lambda}}{m!}  \frac{d^m}{dy^m} (e^{-y}y^{m+\lambda})$$
	for $m=0,1,2,...$ Recall also the Hermite and Laguerre functions 
		$$h_m (y)= \frac{1}{(2^m\sqrt{\pi}m!)^{1/2}} H_m(y) e^{-y^2/2}, \quad 
		\ell_m^{\lambda} (y)=(-1)^m c_m L_m^{\lambda}(y) e^{-y/2},$$
		where $c_m=\sqrt{m!/\Gamma(m+\lambda+1)}$ and $\Gamma$ is the usual Gamma function.
	It is well known that $\{h_m \}_{m=0}^{\infty}$ and $\{\ell_m^{\lambda} \}_{m=0}^{\infty}$ are orthonormal
	bases for $L^2(\R)$ and $L^2(\mathbb{R}_+, y^{\lambda} dy)$, respectively. 
	Finally, ${\cal H}_m=\text{span}\{h_m\}$ and ${\cal L}_m=\text{span}\{\ell_m^{\lambda}\}$.
			
 In this work we restrict ourselves to the study of Toeplitz operators acting 
	on poly-Bergman-type spaces over the two-dimensional Siegel domain $D_{2}$ 
	with Lebesgue measure $d\mu$ (that is, $\lambda=0$). Henceforth,
	the space $\mathcal{A}^{2}_{0L}(D_{2})$  will be simply denoted by $\mathcal{A}^{2}_{L}(D_{2})$;
	similarly, $\ell_m(y)$ and $L_m(y)$ stand for $\ell_m^0(y)$ and $L_m^0(y)$, respectively.
	We will  study Toeplitz operators with nilpotent symbols acting on 
	$\mathcal{A}^{2}_{L}(D_{2})$ in both cases $L=(1,n)$ and $L=(n,1)$. 
	The poly-Bergman-type spaces $\mathcal{A}^{2}_{(1,n)}(D_{2})$ and $\mathcal{A}^{2}_{(n,1)}(D_{2})$   can be identified with $(L^{2}(\R \times \R_{+}))^n$ through 
	a Bargmann type transform (\cite{R-SN-D}).
Several operators are needed to define this unitary map.  To begin,
	we introduce the flat domain $\Do=\C \times \Pi $, where $\Pi= \R \times \R_{+} \subset \C$. 
	Then $\Do$ can be identified with  $D_{2}$ using  the mapping
	\begin{align*}\kappa: \Do & \rightarrow D_2 \\
	(w_{1},w_{2})&  \mapsto (w_{1}, w_{2}+i|w_{1}|^{2}), 
	 \end{align*}
	 whose real Jacobian determinant is everywhere equal to one.
Thus we have the unitary operator
	$U_{0}: L^{2}(D_{2}) \longrightarrow L^{2}(\Do)$
	given by
	$$ (U_{0}f)(w)=f(\kappa(w)).$$
  Take 
	$w=(w_{1}, w_{2}) \in \C \times \Pi$, with $w_{m}=u_m +i v_m$ and  $m=1,2$. We identify 
	$w=(u_1+ iv_1, u_2+iv_2)$ with $(u_1,v_1, u_2,v_2)$. Then
	$$L^{2}(\Do)= L^{2}(\R, du_1)\otimes L^{2}(\R,  dv_{1}) \otimes L^{2}(\R, du_2)\otimes L^{2}(\R_{+}, dv_2).$$
	
We introduce the map
	$$U_{1}= F\otimes I \otimes F \otimes I: L^2(\mathcal{D})\rightarrow L^2(\mathcal{D}),$$
	where $F$ is the Fourier transform acting on $L^{2}(\R)$ by the rule
	$$(Fg)(t)=\frac{1}{\sqrt{2 \pi}} \int_{-\infty}^{\infty} g(x) e^{-itx} dx.$$
	
	Consider now the following two mappings acting on $\Do$:
    \begin{align*}
    \psi_1: \Do & \rightarrow \Do \\
        \xi=(\xi_{1}, t_{2} +i   s_2) & \mapsto (\xi_1, t_{2}+ i\frac{s_{2}}{2|t_{2}|})=w
    \end{align*}	
	and
	\begin{align*}
    \psi_2: \Do & \rightarrow \Do \\
         z=(x_1+iy_1, z_2) & \mapsto \left(\sqrt{|x_{2}|}(x_{1}+y_{1})+ i \frac{1}{2\sqrt{|x_{2}|}}(-x_{1}+y_{1}), z_{2}\right) = \xi,
    \end{align*}	
	where we write $\xi_m=t_m+is_m$  and $z_m=x_m+iy_m$.
	Both functions $\psi_1$ and $\psi_2$ lead to the following unitary operators
	acting on $L^{2}(\Do)$:
	$$(V_1f)(\xi) = \frac{1}{(2|t_{2}|)^{ 1/2}} f(\psi_1(\xi)), \quad (V_{2}g)(z)= g(\psi_2(z)) .$$

	Henceforth  $L=(l_{1}, l_{2})=(j,k)$.
	
	\begin{thm}[\cite{R-SN-D}]\label{puro-iso}
		The operator $U=V_{2}V_{1}U_{1}U_{0}$ 
		is unitary, and maps
		$L^{2}(D_{2}, d\mu)$ onto the space
		$$L^{2}(\Do, d\eta)= L^{2}(\R,dx_1)\otimes L^{2}(\R,dy_1) \otimes L^{2}(\R,dx_2)\otimes L^{2}(\R_{+},dy_2).$$
		For each $L=(1,n) \in \N^{2}$, the operator $U$ restricted to $\mathcal{A}^{2}_{(1,n)}(D_{2})$ is an isometric isomorphism onto
		the space
		$$\mathcal{H}_{(1,n)}^{+}= L^{2}(\R) \otimes \C h_{0}  \otimes L^{2}(\R_{+}) \otimes \Span \{\ell_{0}, ...., \ell_{n-1} \}. $$ 
		For each $L=(n,1) \in \N^{2}$, the operator $U$ restricted to $\mathcal{A}^{2}_{(n,1)}(D_{2})$ is an isometric isomorphism onto
		the space 
		$$\mathcal{H}_{(n,1)}^{+}= L^{2}(\R) \otimes \Span \{h_{0},..., h_{n-1} \}  \otimes L^{2}(\R_{+}) \otimes \C\ell_{0}. $$ 
		Moreover, for each $L=(j,k) \in \N^{2}$, the operator $U$ restricted to the true-Bergman-type space $\mathcal{A}^{2}_{(L)}(D_{2}) = \mathcal{A}^{2}_{((j,k))}(D_{2})$ is an isometric isomorphism onto the space
		$$\mathcal{H}_{(L)}^{+}= L^{2}(\R) \otimes \C h_{j-1} \otimes L^{2}(\R_{+}) \otimes \C \ell_{k-1}. $$ 
	\end{thm}

	\vspace{.3cm}
	We introduce the isometric linear embedding
	$R_{0(L)}: L^{2}(\R \times \R_{+}) \longrightarrow L^{2}(\Do)$
	defined  by
	\begin{equation*}
		(R_{0(L)}g)(x_1,y_1,x_2,y_2)= \chi_{\R_{+}}(x_2)\, g(x_{1}, x_{2}) \, h_{j-1}(y_{1}) \ell_{k-1}(y_{2}).
	\end{equation*}
	Of course $\mathcal{H}_{(L)}^{+}$ is the image of $R_{0(L)}$, and it is also
	the image of $\mathcal{A}^{2}_{(L)}(D_{2})$ under $U$. Thus, the operator
	$$R_{(L)}=R_{0(L)}^{*}U: L^{2}(D_{2}) \longrightarrow L^{2}(\R \times \R_{+}),$$
	unitarily maps the true-poly-Bergman-type space $\mathcal{A}^{2}_{(L)}(D_{2})$ onto  $L^{2}(\R \times \R_{+})$.
	Therefore, $R_{(L)}R_{(L)}^{*}=I$ and $R_{(L)}^{*}R_{(L)}=B_{(L)}$, where
	$B_{(L)}$ is the orthogonal projection from $L^{2}(D_{2})$ onto $\mathcal{A}^{2}_{(L)}(D_{2})$.
	In addition, the operator $R_{(L)}^{*}=U^*R_{0(L)}$ plays the role of the Segal-Bargmann transform 
	for the true-poly-Bergman-type space $\mathcal{A}^{2}_{(L)}(D_{2})$,
	where the adjoint operator
	$R_{0(L)}^{*}: L^{2}(\Do) \longrightarrow L^{2}(\R \times \R_{+})$
	is given by
	\begin{equation*}
	(R_{0(L)}^*f)(x_{1}, x_{2})
	= 
	\int_{\R}\int_{\R_{+}} h_{j-1}(y_{1}) \ell_{k-1}(y_{2}) f(x_{1},y_{1}, x_{2}, y_{2})  dy_{2} dy_{1}
	\end{equation*}
	for all $(x_1,x_2)\in \mathbb{R} \times \mathbb{R}_+$.

Similarly,  we introduce the following isometric linear embeddings
$$R_{0(1,n)},\ R_{0(n,1)} : (L^{2}(\R \times \R_{+}))^{n} \longrightarrow L^{2}(\Do)$$
defined by 
\begin{align*}
	(R_{0(1,n)}g)(x_1,y_1,x_2,y_2) & =  \chi_{\R_{+}}(x_2)\, \, h_0(y_{1}) [N(y_2)]^T g(x_1,x_2) \\
	(R_{0(n,1)}g)(x_1,y_1,x_2,y_2) & = \chi_{\R_{+}}(x_2)\, \ell_{0}(y_{2}) [H(y_1)]^T g(x_1,x_2),
\end{align*}
where $g=(g_1,...., g_n)^T \in(L^{2}(\R \times \R_{+}))^{n}$ and 
$$H(y_1)=(h_{0}(y_1),...,h_{n-1}(y_1))^T \text{ and }
N(y_2)=(\ell_{0}(y_2),...,\ell_{n-1}(y_2)^T.$$  
Clearly, $\mathcal{H}_{(1,n)}^{+}$ and  $\mathcal{H}_{(n, 1)}^{+}$ are the ranges of
$R_{0(1,n)}$ and $R_{0(n,1)}$ respectively;  they are also
the images of $\mathcal{A}^{2}_{(1,n)}(D_{2})$ and $\mathcal{A}^{2}_{(n,1)}(D_{2})$ under $U$.
Consequently, the operators
$$R_{(1,n)}=R_{0(1,n)}^*U,\quad  R_{(n,1)}=R_{0(n,1)}^*U: L^{2}(D_{2}) \longrightarrow 
(L^{2}(\R \times \R_{+}))^n,$$
isometrically map the poly-Bergman-type spaces $\mathcal{A}^{2}_{(1,n)}(D_{2})$ and $\mathcal{A}^{2}_{(n,1)}(D_{2})$ onto  $(L^{2}(\R \times \R_{+}))^n$.
Therefore, $R_{(1,n)}R_{(1,n)}^{*}=I$, $R_{(1,n)}^{*}R_{(1,n)}=B_{(1,n)}$,
$R_{(n,1)}R_{(n,1)}^{*}=I$, and $R_{(n,1)}^{*}R_{(n,1)}=B_{(n,1)}$, where
$B_{(1,n)}$ and $B_{(n,1)}$ are the orthogonal projections from $L^{2}(D_{2})$ onto $\mathcal{A}^{2}_{(1,n)}(D_{2})$ and  $\mathcal{A}^{2}_{(n,1)}(D_{2})$, respectively.
In addition, the operators $R_{(1,n)}^{*}=U^*R_{0(1,n)}$  and $R_{(n,1)}^{*}=U^*R_{0(n,1)}$ play the role of the Segal-Bargmann transform for the poly-Bergman-type spaces $\mathcal{A}^{2}_{(1,n)}(D_{2})$ and  $\mathcal{A}^{2}_{(n,1)}(D_{2})$,
where the adjoint operators
$R_{0(1,n)}^{*},\ R_{0(n,1)}^{*}: L^{2}(\Do) \longrightarrow (L^{2}(\R \times \R_{+}))^{n}$
are given by

\begin{align*}
	(R_{0(1,n)}^*f)(x_{1}, x_{2})
	& = 
	\int_{\R}\int_{\R_{+}} f(x_{1},y_{1}, x_{2}, y_{2}) h_{0}(y_{1}) N(y_2) dy_{2} dy_{1} \\
	(R_{0(n,1)}^*f)(x_{1}, x_{2})
	& = 
	\int_{\R}\int_{\R_{+}}  f(x_{1},y_{1}, x_{2}, y_{2}) \ell_{0}(y_{2}) H(y_1) dy_{2} dy_{1},
\end{align*}
where $(x_1,x_2)\in \mathbb{R} \times \mathbb{R}_+$.\\

	
	\section{Toeplitz operators on the poly-Bergman space  $\mathcal{A}^{2}_{(1,n)}(D_{2})$.} \label{OT}
	
	In this section we study Toeplitz operators with nilpotent symbols acting on the poly-Bergman-type space
	$\mathcal{A}^{2}_{(1,n)}(D_{2})$. In \cite{LibroVasilevski} the author thoroughly developed the theory of Toeplitz operators
	on Bergman spaces, and the author's techniques can be applied to the study of Toeplitz operators
	acting on $\mathcal{A}^{2}_{(1,n)}(D_{2})$. To begin with, a function 
	$\tilde{c} \in L^{\infty}(D_{2}, d\mu)$ is said to be  a nilpotent symbol if it has the form 
	$\tilde{c}(\zeta_{1}, \zeta_{2})=c(\im \zeta_{1},\im \zeta_{2}- |\zeta_{1}|^{2})$, where $c:\R \times \R_+\rightarrow \C$. 
	Then the Toeplitz operator acting on $\mathcal{A}^{2}_{(1,n)}(D_{2})$
	with  nilpotent symbol $\tilde{c}$ is defined by
	$$(T_{c}f)(\zeta)=(B_{(1,n)}(\tilde{c} f))(\zeta),$$
	where $B_{(1,n)}$ is the orthogonal projection from $L^{2}(D_{2})$ onto $\mathcal{A}^{2}_{(1,n)}(D_{2})$.  If we define
	\begin{align*}
	M_f:L^2(D_2)& \rightarrow L^2(D_2) \\
	   g & \mapsto fg,
	\end{align*}
	then $T_{c} = B_{(1,n)}M_{\tilde{c}}$.
	
	The  Bargmann-type operator  $R_{(1,n)}$ identifies the space $\mathcal{A}^{2}_{(1,n)}(D_{2})$ with
	$(L^{2}(\R\times \R_{+}))^n$, and it fits properly in the study of the Toeplitz operators $T_{\widetilde{c}}$.

	\begin{thm}\label{equiv-unitaria}
		Let $\tilde{c}$ be a nilpotent symbol.
		Then the Toeplitz operator $T_{c}$ is unitary equivalent to the multiplication operator $M_{\gamma^c}$, and in fact, $M_{\gamma^c}=R_{(1,n)}T_{c}R_{(1,n)}^*$,  
		where the matrix-valued function $\gamma^{c}:\R \times \R_{+} \rightarrow M_n(\C)$ is given by
		\begin{equation}\label{gamma-c}
\gamma^{c}(x_1,x_2)=
\int_{\R}\int_{\R_{+}} c\left(\frac{-x_{1}+y_{1}}{2\sqrt{x_{2}}} ,\frac{y_{2}}{2x_{2}}\right) (h_0(y_{1}))^{2}
N(y_{2})[N(y_{2})]^T dy_{2}dy_{1}.
		\end{equation}
	\end{thm}

	\begin{proof}
		We have
			\begin{eqnarray*}
			R_{(1,n)}T_{c}R_{(1,n)}^*
			&=& R_{(1,n)}B_{(1,n)} M_{\tilde{c}} R_{(1,n)}^*\\
			&=& R_{(1,n)}R_{(1,n)}^{*}R_{(1,n)} M_{\tilde{c}} R_{(1,n)}^*\\
			&=&R_{(1,n)} M_{\tilde{c}} R_{(1,n)}^*\\
			&=&R_{0(1,n)}^{*}V_{2}V_{1}U_{1}U_{0} (M_{\tilde{c}}) U_0^{-1}U_1^{-1}V_{1}^{-1}V_{2}^{-1} R_{0(1,n)}.
		\end{eqnarray*}
	
	Recall that  $\zeta=\kappa(w)=(w_{1}, w_2 + i |w_{1}|^{2})$, where $w=(w_{1}, w_2) \in \Do$ and $w_m=u_m +iv_m$.
		For $f \in L^{2}(\Do)$,
		$$(U_{0} M_{\tilde{c}} U_{0}^{-1}f)(w)= \tilde{c}(\kappa(w))(U_{0}^{-1}f)(\kappa(w))= \tilde{c}(\kappa(w))f(w).$$
		That is, $U_{0} M_{\tilde{c}} U_{0}^{-1}=M_{\tilde{c}\circ\kappa}$.
		It is easy to see that
		$U_{1} M_{\tilde{c}\circ\kappa} U_{1}^{-1}= M_{\tilde{c}\circ\kappa}.$  Furthermore,
		$$V_{1}M_{\tilde{c}\circ\kappa}V_{1}^{-1}= M_{\tilde{c}\circ\kappa\circ\psi_1}$$ 
		and
		\[
		   V_{2}V_1M_{\tilde{c}\circ\kappa} V_1^{-1} V_{2}^{-1} = M_{\tilde{c}\circ\kappa\circ\psi_1\circ\psi_2}.
		\]
		Thus
		$$M_{\tilde{c}\circ\kappa\circ\psi_1\circ\psi_2}f(\xi) = c\left( \frac{-x_{1}+y_{1}}{2\sqrt{|x_{2}|}} ,\frac{y_{2}}{2|x_{2}|}\right)  f(\xi),$$
		where $\xi = (x_1 +iy_2, x_2 +iy_2)\in \Do$.
		Finally, let $g=(g_1,..., g_n)^T \in(L^{2}(\R \times \R_{+}))^{n}$ and $A=(R_{(1,n)}T_cR_{(1,n)}^* g)(x_1, x_2)$.
		Then
			\begin{eqnarray*}
			A&=& (R_{0(1,n)}^{*}c\left( \frac{-x_{1}+y_{1}}{2\sqrt{|x_{2}|}}  ,\frac{y_{2}}{2|x_{2}|}\right) R_{0(1,n)}g)(x_1, x_2)\\
			&=&	\int_{\R}\int_{\R_{+}} c\left( \frac{-x_{1}+y_{1}}{2\sqrt{|x_{2}|}}  ,\frac{y_{2}}{2|x_{2}|}\right) (R_{0(1,n)}g)(x_{1},y_{1}, x_{2}, y_{2}) h_{0}(y_{1})  N(y_2) dy_{2} dy_{1}\\
			&=& \int_{\R}\int_{\R_{+}} c\left( \frac{-x_{1}+y_{1}}{2\sqrt{|x_{2}|}}  ,\frac{y_{2}}{2|x_{2}|}\right) h_0(y_{1}) [N(y_2)]^T g(x_1,x_2) h_{0}(y_{1}) N(y_2) dy_{2} dy_{1}\\
			&=&  \int_{\R}\int_{\R_{+}}  c\left( \frac{-x_{1}+y_{1}}{2\sqrt{|x_{2}|}}  ,\frac{y_{2}}{2|x_{2}|}\right)  
			(h_{0}(y_{1}))^2 N(y_2)  [N(y_2)]^T g(x_1,x_2) dy_{2} dy_{1}.
		\end{eqnarray*}
		Thus
		$R_{(1,n)}T_{c}R_{(1,n)}^*= M_{\gamma^{c}}$, 
		where $\gamma^{c}(x_1,x_2)$ is given in (\ref{gamma-c}).
	\end{proof}

Studying the full $C^*$-algebra generated by all Toeplitz operators with nilpotent symbols is a difficult task due to the fact that its spectrum is too large.  For this reason we study Toeplitz operators in special cases. In particular, we consider two specific cases of nilpotent symbols. 
First, we study the Toeplitz operators with symbols of the form
$\tilde{b}(\zeta) = b(\im \zeta_{2}- |\zeta_{1}|^{2})$, for which
\begin{equation}\label{sim-b-1}
	\gamma^{b}(x_{1}, x_{2})= \gamma^{b}( x_{2})= \int_{\R_{+}} b\left(\frac{y_{2}}{2x_{2}}\right) N(y_2)  [N(y_2)]^T  dy_{2}.
\end{equation}
	Secondly, we analyze Toeplitz operators with symbols of the form	$\tilde{a}(\zeta) = a(\im \zeta_{1})$,
	for which
\begin{equation}\label{sim-a-2}
	\gamma^{a}(x_{1}, x_{2}) = \int_{\R} a\left( \frac{-x_{1}+y_{1}}{2\sqrt{x_{2}}} \right) 
	(h_0(y_{1}))^{2} dy_{1} I_{n\times n}.\\
\end{equation}

	
	\subsection{Toeplitz operators with symbols $\tilde{b}(\zeta) = b(\im \zeta_{2}- |\zeta_{1}|^{2})$.}\label{seccion-sb}
	
	In this section, we study the $C^*$-algebra generated by all Toeplitz operators $T_b$ with symbols of the form $\tilde{b}(\zeta) = b(\im \zeta_{2}- |\zeta_{1}|^{2})$, where $b:\R_+\rightarrow \C$ has limits at $0$ and $+\infty$.
Under this continuity condition,  we will see that $\gamma^{b}$  is continuous on 
	$\overline{\Pi}:=\overline{\R} \times \overline{\R}_+$, where  
	$\overline{\R}=[-\infty, +\infty]$ and $\overline{\R}_+= [0, +\infty]$ are the two-point compactification of $\R=(-\infty, +\infty)$ and $\R_+=(0,+\infty)$, respectively. The spectral function $\gamma^{b}=(\gamma^b_{jk}):\R\times\R_+\rightarrow M_n(\C)$
	is continuous if all of its matrix entries are continuous.  These matrix entries are given by	
	\begin{equation}\label{sim-b-2}
		\gamma_{jk}^{b}(x_{1}, x_{2}) = \int_{\R_{+}} b\left(\frac{y_{2}}{2x_{2}}\right)
		\ell_{j-1}(y_{2})\ell_{k-1}(y_{2})  dy_{2} , \quad j,k=1,..., n.\\
	\end{equation} 

	Apply the change of variable $y_2 \mapsto 2x_2y_2$ in the integral representation of $\gamma_{jk}^{b}$ .
	Then
	$$\gamma_{jk}^{b}(x_{1}, x_{2})=\gamma_{jk}^{b}(x_{2})= 2x_{2} \int_{\R_{+}} b(y_{2})   \ell_{j-1} (2x_{2}y_{2})\ell_{k-1} (2x_{2}y_{2}) dy_{2}.$$
	We sometimes think of $\gamma_{jk}^{b}$ as a function from $\R_+$ to $\C$, as it depends only  on the variable $x_2$, and is continuous on $\R_+$ because of the
	continuity of  $\ell_{j-1}(y)\ell_{k-1}(y)$  and the Lebesgue dominated convergence theorem.

	Let $L^{\infty}_{\{0, +\infty\}}(\R_{+})$ denote the subspace of $L^{\infty}(\R_{+})$ consisting of all functions having limit values at
 $0$ and $+\infty$.
	For $b \in L^{\infty}_{\{0, +\infty\}}(\R_{+})$, define
	\begin{equation*}
		b_{0}:=\lim_{y\rightarrow 0^{+}} b(y), \quad b_{\infty}:=\lim_{y\rightarrow +\infty} b(y).
	\end{equation*}

The form of the matrix-valued function $\gamma^b$ was obtained in \cite{R-SN} as the spectral function of a Toeplitz operator
acting on  poly-Bergman spaces of the upper half-plane with vertical symbols, i.e., symbols depending only on the imaginary part of $z$.  Thus, we have at least two scenarios in which $\gamma^b$ appears
as a spectral matrix-valued function.

\begin{lem}[\cite{R-SN}]\label{lim-sym-b}
	Let $b \in L^{\infty}_{\{0, +\infty\}}(\R_{+})$. 
	Then the spectral matrix-valued function $\gamma^{b}:\R_+\rightarrow M_n(\C)$ satisfies
	$$b_{\infty} I =\lim_{x_{2}\rightarrow 0^{+}} \gamma^{b}(x_{2}), \quad b_{0} I =\lim_{x_{2}\rightarrow + \infty} \gamma^{b}(x_{2}).$$
\end{lem}

\vspace{.3cm}
Let $M_n(C([0, \infty]))=M_n(\C)\otimes C([0, \infty])$, where $M_n(\C)$ is the algebra of all $n \times n$ matrices with complex entries. 
Let $\mathfrak{C}$ be the $C^*$-subalgebra of $M_n(C([0, \infty]))$ given by
$$\mathfrak{C}= \{M \in  M_n(C([0, \infty])) : M(0), M(+\infty) \in \C I\}.$$

By Lemma \ref{lim-sym-b} and Theorem 4.8 in \cite{R-SN}, we have the following 

\begin{thm}
For all $b\in L^{\infty}_{\{0, +\infty\}}(\R_{+})$, the spectral matrix-valued function $\gamma^{b} $ belongs to the $C^*$-algebra $\mathfrak{C}$. Moreover, the $C^{*}$-algebra generated by all the matrix-valued functions $\gamma^{b}:\overline{\Pi} \rightarrow M_n(\C)$, with $b\in L^{\infty}_{\{0, +\infty\}}(\R_{+})$,  is  equal to $\mathfrak{C}$. That is, the  $C^*$-algebra generated by all the Toeplitz operators $T_{b}$,
	with $b \in L^{\infty}_{\{0, +\infty\}}(\R_{+})$,
	 is isomorphic to $\mathfrak{C}$, where the  isomorphism is defined on the generators by 
	$$T_{b}  \longmapsto \gamma^b.$$
\end{thm}

\vspace{.3cm}
Obviously, in this context, the spectral matrix-valued function $\gamma^b$ is defined and continuous on $\overline{\Pi}$, but 
it is constant along each horizontal straight line. 
Thus,  $\gamma^b$ is identified with a continuous function on $\overline{\R}_+$.


	\subsection{Toeplitz operators with continuous symbols $a(\im \zeta_{1})$.}\label{Alg-a}

Our next stage is the study of the $C^*$-algebra generated by all Toeplitz operators  $T_a$ acting on the poly-Bergman space $\mathcal{A}^{2}_{(1,n)}(D_{2})$,
with symbols pf the form $\tilde{a}(\zeta) = a(\im \zeta_{1})$. 
Recall that $\gamma^a$ is given by
$$\gamma^{a}(x_{1}, x_{2}) = \int_{\R} a\left( \frac{-x_{1}+y_{1}}{2\sqrt{x_{2}}} \right) 
(h_0(y_{1}))^{2} dy_{1} I_{n\times n}$$
for all $(x_1,x_2)\in\Pi=\R\times\R_+$.
It is easy to see that $\gamma^a$ is continuous on $\Pi$.
Note that $\gamma^a$ can be identified with the scalar function 
\begin{align*}
       \Pi & \rightarrow \C \\
(x_1,x_2) & \mapsto \int_{\R} a\left(\frac{-x_{1}+y_{1}}{2\sqrt{x_{2}}} \right) (h_0(y_{1}))^{2} dy_{1}.
\end{align*}
This  function  was obtained in  \cite{Y-R} as the spectral function of the Toeplitz operator acting on the Bergman space $\mathcal{A}^{2}(D_{2})$ with symbol $\tilde{a}(\zeta) = a(\im \zeta_{1})$. 
Based on the results obtained in \cite{Y-R}, we state the following theorem in the context of Toeplitz operators acting on $\mathcal{A}^{2}_{(1,n)}(D_{2})$.

\begin{thm}
	The $C^{*}$-algebra generated by all Toeplitz operators of the form $T_{a}$, where  $\tilde{a}(\zeta) = a(\im \zeta_1)$ for some $a \in C(\overline{\R})$, is isomorphic and isometric to the $C^{*}$-algebra $C(\triangle)$,  where the  quotient space $\triangle=\overline{\Pi} / (\overline{\R} \times \{+\infty\})$  is defined by the identification of $\overline{\R} \times \{\infty\} $ to a point. Futhermore, the $C^{*}$-algebra generated by all  Toeplitz operators $T_{a}$ with $a \in PC(\overline{\R}, \{0\})$,
	 is isomorphic to the $C^{*}$-algebra $C(\overline{\Pi})$, where $PC(\overline{\R}, \{0\})$ consists of all functions continuous on $\overline{\R}\setminus \{0\}$  and having one-sided limits at $0$.
\end{thm}


\section{Toeplitz operators on the poly-Bergman space  $\mathcal{A}^{2}_{(n,1)}(D_{2})$.} \label{cap-n1}
\label{symb-discont}

In this section we study certain $C^*$-algebras generated by Toeplitz operators with nilpotent symbols acting on the poly-Bergman-type space $\mathcal{A}^{2}_{(n,1)}(D_{2})$. 
We apply techniques as in Section \ref{OT}.
A Toeplitz operator acting on $\mathcal{A}^{2}_{(n,1)}(D_{2})$ with nilpotent symbol 
$\tilde{c}(\zeta_{1}, \zeta_{2})=c(\im \zeta_{1},\im \zeta_{2}- |\zeta_{1}|^{2})$  is defined by
$$(T_{c}f)(\zeta)=(B_{(n,1)}(\tilde{c} f))(\zeta),$$
where $B_{(n,1)}$ is the orthogonal projection from $L^{2}(D_{2})$ onto $\mathcal{A}^{2}_{(n,1)}(D_{2})$.
The  Bargmann-type operator  $R_{(n,1)}$ identifies the space $\mathcal{A}^{2}_{(n,1)}(D_{2})$ with
$(L^{2}(\R\times \R_{+}))^n$.

\begin{thm}\label{equiv-unitaria-4}
	Let $\tilde{c}$ be a nilpotent symbol.
	Then the Toeplitz operator $T_{c}$ is unitary equivalent to the multiplication operator $\gamma^{c}I=R_{(n,1)}T_{c}R_{(n,1)}^*$,  
	where the matrix-valued function $\gamma^{c} :\R \times \R_{+} \rightarrow M_n(\C)$ is given by
	\begin{equation}\label{gamma-c-4}
	\gamma^{c}(x_1,x_2)=
	\int_{\R}\int_{\R_{+}} c\left(\frac{-x_{1}+y_{1}}{2\sqrt{x_{2}}} ,\frac{y_{2}}{2x_{2}}\right) (\ell_0(y_{2}))^{2}
	H(y_{1})[H(y_{1})]^T dy_{2}dy_{1}.
	\end{equation}
\end{thm}

\begin{proof}
	We have
	\begin{eqnarray*}
	R_{(n,1)}T_{c}R_{(n,1)}^*
	&=& R_{(n,1)}B_{(n,1)} (M_{\tilde{c}}) R_{(n,1)}^*\\
	&=& R_{(n,1)}R_{(n,1)}^{*}R_{(n,1)} M_{\tilde{c}} R_{(n,1)}^*\\
	&=&R_{(n,1)} M_{\tilde{c}} R_{(n,1)}^*\\
	&=&R_{0(n,1)}^{*}V_{2}V_{1}U_{1}U_{0} M{\tilde{c}} U_0^{-1}U_1^{-1}V_{1}^{-1}V_{2}^{-1} R_{0(n,1)},
\end{eqnarray*}
	where
	$$V_{2}V_1U_1U_0(M{\tilde{c}})U_0^* U_1^* V_1^{-1} V_{2}^{-1}f(\xi) = c\left( \frac{-x_{1}+y_{1}}{2\sqrt{|x_{2}|}} ,\frac{y_{2}}{2|x_{2}|}\right) f(\xi)$$
for all $f\in L^2(\Do)$, where $\xi = (x_1 + iy_1,x_2+iy_2)\in\Do$.

	Finally, let $g=(g_1,..., g_n)^T \in(L^{2}(\R \times \R_{+}))^{n}$ and $B=(R_{(n,1)}T_{c}R_{(n,1)}^* g)(x_1, x_2)$,
	then
		\begin{eqnarray*}
		B&=& (R_{0(n,1)}^{*}c\left( \frac{-x_{1}+y_{1}}{2\sqrt{|x_{2}|}}  ,\frac{y_{2}}{2|x_{2}|}\right) R_{0(n,1)}g)(x_1, x_2)\\
		&=&	 \int_{\R}\int_{\R_{+}} c\left( \frac{-x_{1}+y_{1}}{2\sqrt{|x_{2}|}}  ,\frac{y_{2}}{2|x_{2}|}\right) (R_{0(n,1)}g)(x_{1},y_{1}, x_{2}, y_{2})  \ell_0(y_{2}) H(y_1) dy_{2} dy_{1}\\
		&=& \int_{\R}\int_{\R_{+}}  c\left( \frac{-x_{1}+y_{1}}{2\sqrt{|x_{2}|}}  ,\frac{y_{2}}{2|x_{2}|}\right) 	\ell_{0}(y_{2}) [H(y_1)]^T g(x_1,x_2) \ell_0(y_{2}) H(y_1) dy_{2} dy_{1}\\
		&=&  \int_{\R}\int_{\R_{+}}  c\left( \frac{-x_{1}+y_{1}}{2\sqrt{|x_{2}|}}  ,\frac{y_{2}}{2|x_{2}|}\right) 
		(\ell_0(y_2))^2 H(y_1)  [H(y_1)]^T g(x_1,x_2) dy_{2} dy_{1}.
	\end{eqnarray*}
	Thus
	$R_{(n,1)}T_{c}R_{(n,1)}^*= \gamma^{c} I,$
	where $\gamma^{c}(x_1,x_2)$ is given in (\ref{gamma-c-4}).
\end{proof}

\vspace{.3cm}
As in Section \ref{OT}, we consider two specific cases of nilpotent symblos. Firstly, we will take Toeplitz operators with 
symbols of the form $\tilde{b}(\zeta) = b(\im \zeta_{2}- |\zeta_{1}|^{2})$, for which
\begin{equation}\label{sim-b-3}
	\gamma^{b}(x_{1}, x_{2})= \gamma^{b}( x_{2})= \int_{\R_{+}} b\left(\frac{y_{2}}{2x_{2}}\right) 
	(\ell_0(y_2))^2  dy_{2} I_{n\times n}.
\end{equation}
This spectral function can be identified with the scalar function 
\begin{align*}
\R_+& \rightarrow\C \\
x_2 & \mapsto
\int_{\R_{+}} b\left(\frac{y_{2}}{2x_{2}}\right)  (\ell_0(y_2))^2  dy_{2},
\end{align*} which was studied in \cite{R-SN}.  Thus, the algebra generated by Toeplitz operators of the form $T_{b}$, where $b\in L^{\infty}_{\{0, +\infty\}}(\R_{+})$, has been completely described.

Secondly, we analyze  Toeplitz operators with symbols of the form $\tilde{a}(\zeta) = a(\im \zeta_{1})$, in this case we have
\begin{equation}\label{sim-a-3}
	\gamma^{a}(x_{1}, x_{2}) = \int_{\R} a\left( \frac{-x_{1}+y_{1}}{2\sqrt{x_{2}}} \right) 
	H(y_{1})[H(y_{1})]^T  dy_{1} .\\
\end{equation}
From this point on, we focus on describing the $C^*$-algebra generated by matrix-valued functions of this type.


\subsection{Continuity of the spectral function $\gamma^a$.}\label{con-func}

In order to describe the $C^*$-algebra generated  by Toeplitz operators acting on $\mathcal{A}^{2}_{(n,1)}(D_{2})$ with nilpotent symbols of the form $\tilde{a}(\zeta) = a(\im \zeta_{1})$, first we will analyze the continuous extension of 
$\gamma^{a}=(\gamma_{jk}^a)$ to the compactification $\overline{\Pi}:=\overline{\R} \times \overline{\R}_+$.
Make the change of variable $y_1 \mapsto 2\sqrt{x_2} y_1+x_1$ in the integral representation of  $\gamma_{jk}^{a}$, then
$$\gamma^{a}_{jk}(x_{1}, x_{2})= 2\sqrt{x_{2}} \int_{\R} 
a(y_{1})h_{j-1}(2\sqrt{x_{2}}y_{1}+x_{1})h_{k-1}(2\sqrt{x_{2}}y_{1}+x_{1})  dy_{1}.$$
The function $\gamma_{jk}^a$ is continuous at each point $(x_{1},x_{2}) \in \Pi$ by the continuity of
$h_{j-1}h_{k-1}$ and the Lebesgue dominated convergence theorem.
Next, we will prove that  $\gamma_{jk}^a$ has a one-sided  limit at each point of $\R \times \{0\}$.
For  $a \in L^{\infty}(\R)$ we introduce the notation
\begin{equation}\label{limit-a-infty}
	a_{-}= \lim_{y\rightarrow -\infty} a(y) \quad \text{and} \quad 	a_{+}=\lim_{y\rightarrow +\infty} a(y)
\end{equation}
if such limits exist.

\begin{lem}\label{lim-gamma-x0-0}
	Let $a\in L^{\infty}(\R)$, and suppose that $a$ has limits at  $\pm\infty$.
	Then for each $x_{0} \in \R$, the spectral matrix-valued function $\gamma^{a}:
    \Pi \rightarrow M_n(C)$  satisfies
	\begin{align}\label{limit-R}
	\begin{split}
	\lim_{(x_{1}, x_{2}) \rightarrow (x_{0}, 0)} \gamma^{a}(x_{1}, x_{2}) =\ & a_{-}\int_{-\infty}^{x_{0}} 
	H(y_1)[H(y_1)]^T dy_{1}  \\
	& 
	+  a_{+}\int_{x_{0}}^{+\infty} H(y_1)[H(y_1)]^T  dy_{1} .
	\end{split}
	\end{align}
\end{lem}

\begin{proof}
Let $A$ denote the $(j, k)$-entry of the right-hand side of (\ref{limit-R}).
	Take $\epsilon > 0$. We will prove that there exists $\delta >0$ such that
	$|\gamma_{}^{a}(x_{1}, x_{2})-A| < \epsilon$ whenever $|x_{1}-x_{0}|< \delta$ and $0<x_{2}< \delta$.
	Note that $|a_-|,|a_+| \leq \|a\|_{\infty}$.
	Since $C_{jk} = \int_{-\infty}^{\infty} |h_{j-1}(y_{1}) h_{k-1}(y_{1})| dy_1 >0$,
	there exists  $\delta_{1} >0$  such that
		$$\|a\|_{\infty} \int_{-\delta_{1}+x_{0}}^{\delta_{1}+x_{0}}|h_{j-1}(y_{1}) h_{k-1}(y_{1})| dy_{1} < \frac{\epsilon}{5 }.$$
	Then
		\begin{eqnarray*} 
		I
		&:=&|\gamma_{jk}^{a}(x_{1}, x_{2})-A| \\
		& = & \bigg| \int_{-\infty}^{\infty} a\left(\frac{-x_{1}+y_{1}}{2\sqrt{x_{2}}}  \right)  h_{j-1}(y_{1}) h_{k-1}(y_{1})dy_{1}\\ 
		& & -  a_{-}\int_{-\infty}^{x_{0}} h_{j-1}(y_{1}) h_{k-1}(y_{1})dy_{1}
		- a_{+}\int_{x_{0}}^{\infty} h_{j-1}(y_{1}) h_{k-1}(y_{1})dy_{1}\bigg| \\
		& \leq &\int_{-\infty}^{-\delta_{1}+x_{0}}\left| a\left( \frac{-x_{1}+y_{1}}{2\sqrt{x_{2}}} \right) -a_{-}\right|  |h_{j-1}(y_{1}) h_{k-1}(y_{1})|d{y_{1}} \\
		& & + |a_{-}| \int_{-\delta_{1}+x_{0}}^{x_{0}}|h_{j-1}(y_{1}) h_{k-1}(y_{1})|dy_{1} \\
		&& + |a_{+}| \int_{x_{0}}^{\delta_{1}+x_{0}}|h_{j-1}(y_{1}) h_{k-1}(y_{1})| dy_{1}\\
		& & + \int_{-\delta_{1}+x_{0}}^{\delta_{1}+x_{0}} \left| a\left( \frac{-x_{1}+y_{1}}{2\sqrt{x_{2}}} \right)  h_{j-1}(y_{1}) h_{k-1}(y_{1}) \right| d{y_{1}} \\
		& & + \int_{\delta_{1}+x_{0}}^{\infty}\left| a\left( \frac{-x_{1}+y_{1}}{2\sqrt{x_{2}}} \right) -a_{+}\right|  |h_{j-1}(y_{1}) h_{k-1}(y_{1})|d{y_{1}} \\
		& \leq &    C_{jk}  \max_{-\infty< y_{1}< -\delta_{1}+x_{0}}\left| a\left( \frac{-x_{1}+y_{1}}{2\sqrt{x_{2}}} \right) -a_{-}  \right| +
		\frac{ 3\epsilon}{5} \\
		& & + C_{jk} \max_{\delta_{1}+ x_{0}< y_{1}< \infty}\left| a\left( \frac{-x_{1}+y_{1}}{2\sqrt{x_{2}}} \right) -a_{+}  \right|. 
	\end{eqnarray*}
	
We have assumed that $a$ converges  at $\pm \infty$. Thus there exists $N>0$ such that 
	$|a(y)-a_{-}|< \epsilon/(5C_{jk})$ and $|a(y)-a_{+}|< \epsilon/(5C_{jk})$ for $|y|> N$. 
	Let $\delta=\min \{\delta_{1}/2, \delta_1^2/(16N^2) \}$.
	Then we have $\frac{1}{2\sqrt{x_{2}}}|-x_{1}+y_{1}| > N$ if $|x_1-x_0| < \delta$,
	$ 0 <x_2 < \delta$ and $|y_1-x_0|\geq \delta_{1}$. 
	Thus,
	$$\max_{-\infty< y_{1}< -\delta_{1}+x_{0}}\left| a\left(\frac{-x_{1}+y_{1}}{2\sqrt{x_{2}}} \right) -a_{-}  \right|< \frac{\epsilon}{5C_{jk}} $$
	and
	$$\max_{\delta_{1}+ x_{0}< y_{1}< \infty}\left| a\left(\frac{-x_{1}+y_{1}}{2\sqrt{x_{2}}} \right)  -a_{+}  \right| < \frac{\epsilon}{5C_{jk}}.$$
Finally, we conclude that $|\gamma_{jk}^{a}(x_{1}, x_{2})-A|<\epsilon$ whenever 
		$|x_{1}-x_{0}|< \delta$ and $0<x_{2}< \delta$.
\end{proof}

\vspace{.4cm}
In general, the matrix-valued function $\gamma^a$ does not converge at the points  $(\pm \infty, +\infty)\in \overline{\Pi}$;  
however, $\gamma^{a}$  has limit values along the parabolas $x_{2}=\alpha (x_{1}^{2}+1)$, with $\alpha >0$.
For this reason we introduce the mapping $\Phi: \Pi \longrightarrow \Pi$ given by
$$
\Phi(x_{1}, x_{2})=\left(x_{1}, \frac{x_{2}}{x_{1}^{2} + 1}\right).$$
We will prove that
$\phi^{a}=\gamma^a \circ \Phi^{-1}:\Pi\rightarrow M_n(\C)$ has a continuous extension to
$\overline{\Pi}=\overline{\R}\times\overline{\R}_+$ with the usual topology.
It is easy to see that $\Phi^{-1}(t_1,t_2)=(t_1, (t_1^2+1)t_2)$. 
Concerning the spectral properties of $T_a$, the matrix-valued function  $\phi^{a}$ contains the same information as $\gamma^a$, 
but $\phi^a$ behaves much better than $\gamma^a$, at least for $a$ continuous on $\overline{\R}$. From now on we take $\phi^a$ 
as the spectral matrix-valued function for the Toeplitz operator $T_a$. A direct computation shows that
$$\phi^{a}(t_{1}, t_{2})=
\int_{-\infty}^{\infty} a\left(\frac{-t_{1}+s_{1}}{2\sqrt{t_{2}(t_{1}^{2} +1)}}\right) 	
H(s_{1})[H(s_{1})]^T  ds_{1}.$$

Note that both $\Phi$ and $\Phi^{-1}$ are continuous on $\mathbb{R} \times [0,+\infty)$. 
In addition, the spectral function $\phi^{a}=\gamma^{a} \circ \Phi^{-1}$ is continuous on $\mathbb{R} \times [0,+\infty)$ because $\gamma^{a}$ is. Since 
$\Phi^{-1}(t_1, 0)=(t_1,0)$,  we have that $\phi^{a}(t_{1}, 0)=\gamma^{a}(t_1, 0)$ for all $t_1\in\C$.

\begin{thm}
	For $a \in C(\overline{\R})$, the spectral matrix-valued function $\phi^a: \Pi \rightarrow M_n (\mathbb{C})$ 
	can be extended continuously to $\overline{\Pi}=\overline{\R} \times \overline{\R}_+$.
	Furthermore,  $\phi^a$ is constant along $\overline{\mathbb{R}} \times \{+\infty\}$.
\end{thm}

\begin{proof}
The result follows from Lemmas \ref{lim-gamma-x0-0},  \ref{lim-phi-infty-0}, \ref{lim-phi-infty-t0}  and  \ref{lim-phi-t0-infty} below.
\end{proof}

\begin{lem}\label{lim-phi-infty-0}
	Let $a\in L^{\infty}(\R)$, and suppose that $a$ converges at $\pm \infty$. 
	Then $\phi^{a}=(\phi_{jk}^a)$ satisfies
		$$\lim_{(t_{1}, t_{2}) \rightarrow (+\infty, 0)} \phi^{a}(t_{1}, t_{2})= a(-\infty)I.$$
	That is, for $\epsilon >0$ there exists $\delta>0$ and $N>0$ such that $|\phi_{jk}^{a}(t_{1}, t_{2})- \delta_{jk}a(-\infty)|<\epsilon$
	whenever $0< t_{2}< \delta$ and $t_{1}>N$. Analogously,
	$$ \lim_{(t_{1}, t_{2}) \rightarrow (-\infty, 0)} \phi^{a}(t_{1}, t_{2})=a(+\infty)I.$$	
\end{lem}

\begin{proof}
	Suppose that $a(-\infty)=0$. Let  $\epsilon>0$.
	Since $h_{j-1}(s_{1}) h_{k-1}(s_{1} )\in L^{2}(\R)$, there exists $s_{0}>0$ such that
	$$\|a\|_{\infty} \int_{s_{0}}^{\infty}|h_{j-1}(s_{1}) h_{k-1}(s_{1})|  ds_{1} < \frac{\epsilon}{2}.$$
	Let $C_{jk} = \int_{-\infty}^{\infty} |h_{j-1}(s_{1}) h_{k-1}(s_{1})| dy_1 >0$.  Then we have
		\begin{eqnarray*}
		|\phi_{jk}^{a}(t_{1}, t_{2})| 
		& = & \left| \int_{-\infty}^{\infty} a\left(\frac{-t_{1}+s_{1}}{2\sqrt{t_{2}(t_{1}^{2} +1)}}\right) 
		h_{j-1}(s_{1}) h_{k-1}(s_{1}) ds_{1}  \right|\\
		& \leq & \int_{-\infty}^{s_{0}} \left| a\left(\frac{-t_{1}+s_{1}}{2\sqrt{t_{2}(t_{1}^{2} +1)}}\right) 
		h_{j-1}(s_{1}) h_{k-1}(s_{1}) \right| d{s_{1}}\\
		& & + \int_{s_{0}}^{\infty}\left| a\left(\frac{-t_{1}+s_{1}}{2\sqrt{t_{2}(t_{1}^{2} +1)}}\right)  
		h_{j-1}(s_{1}) h_{k-1}(s_{1}) \right| d{s_{1}}\\
		& \leq &   C_{jk}  \max_{-\infty< s_{1}< s_{0}}\left| a\left( \frac{-t_{1}+s_{1}}{2\sqrt{t_{2}(t_{1}^{2} +1)}} \right)  \right| +
		\frac{ \epsilon}{2}.
	\end{eqnarray*}
	
	Since $a$ converges to zero at $-\infty$, there exists $N_{1}>0$ such that $C_{jk}|a(s)|< \epsilon/2$ for $-s> N_{1}$. 
	Take $\delta= 1/(16N_{1}^{2})$. Then we have  $\frac{1}{2\sqrt{t_2}} > 2N_1$ for $0<t_2<\delta$.
	On the other hand, 	assume  $t_{1}> s_0$ 	and $-\infty < s_1 < s_{0}$.  Then
	$$\frac{t_{1}- s_{1}}{\sqrt{t_{1}^{2} +1}} > \frac{t_{1}- s_{0}}{\sqrt{t_{1}^{2} +1}}.$$
	The right-hand side of this inequality converges to $1$ when 
	$t_1$ tends to $+\infty$, thus there exists
	$N_2>s_0$ such that $(t_{1}- s_{0})/\sqrt{t_{1}^{2} +1}>1/2$ for $t_1>N_2$. Consequently,
	$$N_1=2N_1\frac{1}{2} < \frac{1}{2\sqrt{t_2}} \frac{t_{1}- s_{0}}{\sqrt{(t_{1}^{2} +1)}} <
	\frac{t_{1}-s_{1}}{2\sqrt{t_{2}(t_{1}^{2} +1)}}. $$
	For $0<t_2<\delta$ and $t_1>N:=\max\{ s_0,N_2 \}$ we have
	$$C_{jk}\left| a\left(\frac{-t_{1}+s_{1}}{2\sqrt{t_{2}(t_{1}^{2} +1)}}\right)  \right| < \frac{\epsilon}{2}.$$
	
	We define $\hat{a}(s)=a(s)-a_2$ in the case $a(-\infty) \neq 0$, where $a_2:=a(-\infty)$. 
 Note that $\hat{a}$  converges to zero at $-\infty$, 
	and $\phi^{a_1+a_2}=\phi^{a_1}+\phi^{a_2}$ for any nilpotent symbols $a_1$ and $a_2$.
	Then
	\begin{eqnarray*}
	\lim_{(t_{1}, t_{2}) \rightarrow (+\infty, 0)} \phi_{jk}^{a}(t_{1}, t_{2})
	& = & \lim_{(t_{1}, t_{2}) \rightarrow (+\infty, 0)} \phi_{jk}^{\hat{a}+a_2}(t_{1}, t_{2}) \\
	& = &  \lim_{(t_{1}, t_{2}) \rightarrow (+\infty, 0)} \phi_{jk}^{\hat{a}}(t_{1}, t_{2})
	+
	a_2 \int_{-\infty}^{\infty} h_{j-1}(s_{1}) h_{k-1}(s_{1}) ds_{1}\\
	& = & a(-\infty)\delta_{jk}.
\end{eqnarray*}

Finally, the limit of $\phi^a$ at $(-\infty,0)$ can be proved analogously. 
\end{proof}

\begin{lem}\label{lim-phi-infty-t0}
	Let $t_{0} \in \R_+$. If $a\in L^{\infty}(\R)$ is continuous at $-1/(2\sqrt{t_0})$, 
	then  the spectral matrix-valued function $\phi^{a}$ satisfies
	$$\lim_{(t_{1}, t_{2}) \rightarrow (+\infty, t_{0})} \phi^{a}(t_{1}, t_{2})= a\left(-\frac{1}{2\sqrt{t_{0}}}\right)I.$$
	Analogously, if $a$ is continuous at $1/(2\sqrt{t_0})$,  then
	$$ \lim_{(t_{1}, t_{2}) \rightarrow (-\infty, t_{0})} \phi^{a}(t_{1}, t_{2})=a\left(\frac{1}{2\sqrt{t_{0}}}\right)I.$$	
\end{lem}

\begin{proof}
	Suppose that $a$ converges to zero at $-1/(2\sqrt{t_{0}})$. 
	Let $\epsilon >0$.	Since $h_{j-1}(s_{1}), h_{k-1}(s_{1}) \in L^{2}(\R, ds_{1})$, there exists $s_{0}>0$  such that
	\begin{equation} \label{acotar-colas-4}
	\|a\|_{\infty}\int_{-\infty}^{-s_{0}}|h_{j-1}(s_{1}) h_{k-1}(s_{1})|ds_{1} < \frac{\epsilon}{3}, \quad
	\|a\|_{\infty}\int_{s_{0}}^{\infty} |h_{j-1}(s_{1}) h_{k-1}(s_{1})| ds_{1} < \frac{\epsilon}{3}.
\end{equation}
	Take into account $C_{jk}=\int_{-\infty}^{\infty}|h_{j-1}(s_{1}) h_{k-1}(s_{1})|ds_{1}>0$ in the following computation
	\begin{eqnarray*}
	|\phi_{jk}^{a}(t_{1}, t_{2})|
	& \leq & \int_{-\infty}^{-s_{0}} 
	\left| a\left(\frac{-t_{1}+s_{1}}{2\sqrt{t_{2}(t_{1}^{2} +1)}}\right) \right| |h_{j-1}(s_{1}) h_{k-1}(s_{1})|  d{s_{1}} \\
	& & + \int_{-s_{0}}^{s_{0}}\left| a\left(\frac{-t_{1}+s_{1}}{2\sqrt{t_{2}(t_{1}^{2} +1)}}\right) \right|   
	|h_{j-1}(s_{1}) h_{k-1}(s_{1})|d{s_{1}}\\
	& & + \int_{s_{0}}^{\infty} \left| a\left(\frac{-t_{1}+s_{1}}{2\sqrt{t_{2}(t_{1}^{2} +1)}}\right) \right|    
	|h_{j-1}(s_{1}) h_{k-1}(s_{1})| d{s_{1}}\\
	& < &   \frac{2 \epsilon}{3}  +
	C_{jk} \max_{-s_{0}< s_{1}< s_{0}}\left| a\left(\frac{-t_{1}+s_{1}}{2\sqrt{t_{2}(t_{1}^{2} +1)}}\right) \right|. 
\end{eqnarray*}
		
	Because of the continuity of $a$  at $-1/(2\sqrt{t_0})$, there exists $\delta_{1}>0$ such that
	$C_{jk}|a(s)|< \epsilon/3$ for $|s-\frac{-1}{2\sqrt{t_{0}}}|< \delta_{1}$. 
	Let us estimate the value of the  argument of $a$:
	\begin{eqnarray*}
		I&:=& \left|\frac{1}{2\sqrt{t_{2}(t_{1}^{2} +1)}}(-t_{1}+s_{1})- \frac{-1}{2\sqrt{t_{0}}} \right| \\
		& \leq&
		\left|- \frac{1}{2\sqrt{t_{2}}}+ \frac{1}{2\sqrt{t_{0}}} \right|\left| \frac{t_{1}}{\sqrt{t_{1}^{2} + 1}}  \right| 
		+\frac{1}{2\sqrt{t_{0}}} \left| 1- \frac{t_{1}}{\sqrt{t_{1}^{2} + 1}}  \right|  \\
		& & + \left| \frac{s_{1}}{2\sqrt{t_{2}(t_{1}^{2} + 1)}} \right|. 
	\end{eqnarray*}
	Choose $\delta >0$ such that
	$\left|- \frac{1}{2\sqrt{t_{2}}}+ \frac{1}{2\sqrt{t_{0}}} \right| < \delta_{1} /3$ for $|t_{2}- t_{0}|<\delta$. 
	Pick $N_{1} >0$ such that
	$\left|1-\frac{t_{1}}{\sqrt{t_{1}^{2} + 1}}\right| < (2\sqrt{t_{0}}\delta_{1}) /3$  whenever $t_{1} > N_{1}$.
	Now assume that $|t_2-t_0|<\delta$ and $|s_{1}| < s_{0}$. Then $|\frac{1}{2\sqrt{t_{2}}}| < \frac{1}{2\sqrt{t_{0}}} + \frac{\delta_{1}}{3}$. 
	Thus,  $\frac{|s_{1}|}{2\sqrt{t_{2}(t_{1}^{2} + 1 )}}$ 
	converges to $0$ when $t_{1}$ tends to $+\infty$. Therefore, there exists $N>N_1$ such that
	$\frac{|s_{1}|}{2\sqrt{t_{2}(t_{1}^{2} + 1 )}}< \delta_{1} /3$ for	$t_{1}> N$.
	The additional condition $t_1>N$ implies
	$$C_{jk} \left| a\left(\frac{-t_{1}+s_{1}}{2\sqrt{t_{2}(t_{1}^{2} +1)}}\right) \right| < \epsilon /3.$$
	Hence $|\phi_{jk}^{a}(t_{1}, t_{2})| < \epsilon$ if $|t_2-t_0|<\delta$ and $t_1>N$.
	
	If $a$  does not converge to zero at $-\frac{1}{2\sqrt{t_{0}}}$, then take the function
	$\hat{a}(s)=a(s)-a_2$ and proceed as in the proof of Lemma \ref{lim-phi-infty-0}, where $a_2=a\left(-\frac{1}{2\sqrt{t_{0}}}\right)$. 
	
Finally, the justification of the limit  of $\phi^a$ at $(-\infty,t_0)$ can be done analogously. 
\end{proof}

\begin{lem}\label{lim-phi-t0-infty}
	Let $a\in L^{\infty}(\R)$ be continuous at $0\in \R$. For $t_{0} \in \overline{\R}$, 
	 the spectral matrix-valued function  $\phi^{a}$ satisfies
	$$\lim_{(t_{1}, t_{2}) \rightarrow (t_{0}, +\infty)} \phi^{a}(t_{1}, t_{2})= a(0)I.$$
	Actually, we have uniform convergence of $\phi^a(t_1,t_2)$ at $(t_0,+\infty)$; t¨hat is,  for $\epsilon>0$, 
	there exists $N>0$ such that $|\phi_{jk}^{a}(t_{1}, t_{2})-a(0)|<\epsilon$ for all $t_2>N$ and for all $t_1\in \overline{\R}$. 
\end{lem}

\begin{proof}
	Suppose that $a(0)=0$. Let $\epsilon >0$, and choose 
	$s_{0}>0$ such that equations  (\ref{acotar-colas-4}) hold. Let $C_{jk} = \int_{-\infty}^{\infty} |h_{j-1}(s_{1}) h_{k-1}(s_{1}) | ds_{1} >0$.
	Then
	\begin{eqnarray*}
	|\phi_{jk}^{a}(t_{1}, t_{2})|
	& \leq & \int_{-\infty}^{-s_{0}} 
	\left| a\left(\frac{-t_{1}+s_{1}}{2\sqrt{t_{2}(t_{1}^{2} +1)}}\right) \right||h_{j-1}(s_{1}) h_{k-1}(s_{1})| d{s_{1}} \\
	& &+ \int_{-s_{0}}^{s_{0}}\left| a\left(\frac{-t_{1}+s_{1}}{2\sqrt{t_{2}(t_{1}^{2} +1)}}\right) \right|   
	|h_{j-1}(s_{1}) h_{k-1}(s_{1})|d{s_{1}}\\
	& &+ \int_{s_{0}}^{\infty} \left| a\left(\frac{-t_{1}+s_{1}}{2\sqrt{t_{2}(t_{1}^{2} +1)}}\right) \right|    
	|h_{j-1}(s_{1}) h_{k-1}(s_{1})| d{s_{1}}\\
	& < &    \frac{2 \epsilon}{3}  +
	C_{jk} \max_{-s_{0}< s_{1}< s_{0}}\left| a\left(\frac{-t_{1}+s_{1}}{2\sqrt{t_{2}(t_{1}^{2} +1)}}\right) \right|.
\end{eqnarray*}
	
	By the continuity  of $a$ at $0$, there exists $\delta_{1}>0$ such that $|a(s)|< \epsilon/(3C_{jk})$ for 
	$|s| < \delta_{1}$. For  $ -s_0 <s_1 <s_0$ we have
	\begin{eqnarray*}
		\left| \frac{-t_{1}+s_{1}}{2\sqrt{t_{2}(t_{1}^{2} +1)}} \right| &\leq& 
		\frac{1}{2\sqrt{t_{2}}} \left( \left|\frac{t_{1}}{\sqrt{t_{1}^{2} + 1}}\right| +\frac{|s_{1}|}{\sqrt{t_{1}^{2} + 1}} \right)
		 <   \frac{1}{2\sqrt{t_{2}}}(1+s_0).
	\end{eqnarray*}

	Take $N=(1+s_{0})^{2}/(4 \delta_{1}^{2})$. The inequality  $t_2>N$ implies
	$\frac{1}{2\sqrt{t_2}}<\frac{\delta_{1}}{(1+s_0)}$. Thus, 
  if $t_{2} >N$, $t_{1} \in \overline{\R}$ and $-s_{0}<s_1 < s_{0}$, then
	$$\left| \frac{-t_{1}+s_{1}}{2\sqrt{t_{2}(t_{1}^{2} +1)}} \right| < \delta_{1}.$$
Consequently, 	$|\phi_{jk}^{a}(t_{1}, t_{2})|<\epsilon$ for all $t_2>N$ and $t_1 \in \overline{\R}$.
	
	Finally, in the case $a(0) \neq 0$, the proof can be carry out by considering the symbol  $\hat{a}(s)=a(s)-a(0)$.
\end{proof}

For each nilpotent symbol $a\in C(\overline{\R})$, the spectral function $\phi^{a}$ is continuous on $\overline{\Pi}$  and is constant along $\overline{\R}\times \{+\infty\}$.
 In order to obtain a larger algebra, we now consider symbols $a\in PC(\overline\R,\{0\})$, where $PC(\overline{\R}, \{0\})$ is the set of continuous functions on $\overline{\R}$ with one-sided limits at $0$.

	Consider the indicator function
$\chi_+=\chi_{[0, +\infty]}$, for which
\begin{equation}\label{ind.matriz}
	\phi^{\chi_+}(t_{1}, t_{2})=
	\int_{t_1}^{\infty} 	H(s_{1})[H(s_{1})]^T  ds_{1}.
\end{equation}

\begin{thm}
	Let  $a \in PC(\overline{\R}, \{0\})$. Then  the spectral matrix-valued function $\phi^a$ can be  extended continuously to  $\overline{\Pi}$.
\end{thm}

\begin{proof}
		For $a \in PC(\overline{\R},\{0 \})$ we have
	$$a(s)=\hat{a}(s)+[a(0_+)-a(0_-)]\chi_+(s),$$
	where $a(0_-)$ and  $a(0_+)$ are the one-side limits of $a$ at $0$, and $\hat{a}(s)= a(s)+[a(0_-)-a(0_+)]\chi_+(s)$. 
	Since $\hat{a} \in C(\overline{\R})$, the spectral function $\phi^{\hat{a}}$ is continuous on $\overline{\Pi}$.
	According to (\ref{ind.matriz}), $\phi^{\chi_+}$ is obviously continuous on $\overline{\Pi}$. 
\end{proof}

The spectral matrix-valued function $\phi^{\chi_+}$ depends only the real variable $t_1$;
thus it can be identified with the one-variable function
\begin{equation}\label{phi+}
\phi^{+}(t):=\int^{\infty}_{t} 	H(s)[H(s)]^T ds .
\end{equation}

\begin{lem}\label{propiedades}
	The matrix-valued function  $\phi^{+}=(\phi^{+}_{jk})$ satisfies:
	\begin{enumerate}
		\item[(1)]  $\phi^{+}(-\infty)= I $ and   $\phi^{+}(+\infty)= 0$.
		\item[(2)]  For each  $t \in \R$,  $\phi^{+}(t)$ is symmetric positive definite and
		$\|\phi^{+}(t)\| \leq 1$, where $\| \cdot\|$ is the uniform norm.
		\item[(3)] There exists $C \in M_n(\mathbb{C})$ such that for all $t\in \overline{\R}$, one has that $\phi^{+}(t)=C M_{t} C^T$, 
		where $C \in M_n(\mathbb{C})$ and
		$$M_{t}= \int_{t}^{\infty} e^{-s^2} SS^T   ds, \quad  S=(1, s, ..., s^{n-1})^T.¨$$
		\item[(4)] For each $t \in \R$ and  $\lambda \in \C$,  then $\det (\lambda I -  \phi^{+}(t))= 0$ 
		if and only if $\det (\lambda M_{-\infty}- M_{t})=0$.
	\end{enumerate}
\end{lem}

\begin{proof}
	Part (1) follows since $\{h_j\}_{j=0}^{\infty}$ is an orthonormal basis for $L_2(\mathbb{R})$.
	The matrix $\phi^{+}(t) $ is symmetric for all $t$ because of $H(s) H(s)^T$ is symmetric for all $s$.
	Let $v \in \C^n$ be an unit vector, then
	\begin{equation}\label{positiva}
		\langle\phi^{+}(t) v,v \rangle =\int_{t}^{\infty} |\langle H(s), v \rangle|^2 ds,
	\end{equation}
	where $e^{s^2}|\langle H(s), v \rangle|^2$ is a nonzero polynomial,
	thus $\langle\phi^{+}(t) v,v \rangle>0$.
	Now, we note that 
	\begin{equation*}
		\langle\phi^{+}(t) v,v \rangle 
		<\int_{-\infty}^{\infty} |\langle H(s), v \rangle|^2 ds = \langle\phi^{+}(-\infty) v,v \rangle 
		=\langle I v,v \rangle =1,
	\end{equation*}
	hence $\|\phi^{+}(t)\| \leq 1$. This proves $(2)$.
	
	The Hermite function is given by
	\begin{align*}
		h_k (s) & = e^{-s^2/2}  \sum_{m=0}^{[k/2]} d_{km} \ s^{k-2m}, \quad 
		d_{km}=  \frac{1}{\sqrt{2^k k!\sqrt{\pi} }}  \frac{(-1)^m k! 2^{k-2m} }{m!(k-2m)!} \\
		& = e^{-s^2/2}  \sum_{m=0}^{k} c_{km} \ s^{m},
	\end{align*}
	and define
	$$C= \left( \begin{array}{ccccc}
		c_{00}     &  0            &   \cdots & 0\\
		c_{10}      &   c_{11}    &   \cdots & 0 \\
		\vdots      & \vdots      &   \ddots& \vdots \\
		c_{n-1,0} & c_{n-1,1} &  \cdots &   c_{n-1,n-1}
	\end{array}
	\right).$$
	Then  $H(s)=(h_0(s), ..., h_{n-1}(s))^T = e^{-s^2 / 2} CS$ and
	$$H(s)H(s)^T = e^{-s^2}CS(CS)^T= e^{-s^2}CSS^T C^T.$$
	Therefore,
	$\phi^{+}(t)= C M_{t} C^T.$
	Also, $\det C \neq 0$ since $C$ is a lower triangular matrix and the scalars $c_{jj}$ are nonzero. This proves $(3)$.
	
	Finally, let $\lambda \in \C$. We have $I=\phi^{+}(-\infty)=CM_{-\infty}C^T$. Then
	\begin{eqnarray*}
		\lambda I - \phi^{+}(t)  &=& \lambda C M_{-\infty} C^T - C M_{t} C^T\\
		&=& C(\lambda M_{-\infty}- M_{t})C^T.
	\end{eqnarray*}
	Thus $\det (\lambda I -  C M_{t} C^T)= 0$ if and only if
	$\det (\lambda M_{-\infty}- M_{t})=0$.
\end{proof}


\subsection{The algebra generated by the Toeplitz operator $T_{\chi^+}$.}\label{separar-arriba}

The $C^*$-algebra generated by $T_{\chi^+}$ is isomorphic to the $C^*$-algebra
generated  by $\phi^+$. Let $\mathcal{D}_n$ be $C^*$-algebra generated by $I$ and $\phi^{+}$, which is a subalgebra of 
$M_n(\mathbb{C})\otimes C(\overline{\mathbb{R}})$, where the metric is given by
$\|M\|=\max_{t\in \overline{\mathbb{R}}} \|M(t)\|$.

According to Lema \ref{propiedades},  the matrix  $\phi^{+}(t)$  is diagonalizable for each $t \in \R$ and its spectrum $\sigma (\phi^{+}(t)) $ lies in $[0, 1]$. 
The eigenvalues are given by  the equation
$\det (\lambda M_{-\infty}- M_{t})=0$.  
There exists an orthogonal matrix $B(t)$ such that 
$$D(t):=B(t)^T \phi^{+}(t) B(t)=\text{diag }\{\lambda_1(t),..., \lambda_n(t)\},$$
that is, 
if $B(t)=[v_1(t)\cdots v_n(t)]$, then $\phi^{+}(t) v_j(t)=\lambda_j(t) v_j(t) $ for $j=1,...,n$.
We may assume that $B$ and $\lambda_j$ are continuous on $\overline{\mathbb{R}}$,
and $\lambda_1(t)\leq \lambda_2(t) \leq \cdots \leq \lambda_n(t)$.
We have $\lambda_j(-\infty)=1$ and $\lambda_j(+\infty)=0$.

Up to isomorphism, the $C^*$-algebra $\mathcal{D}_n$ is equal to the $C^*$-algebra generated by
$D$, that is, each elements $\varphi \in \mathcal{D}_n$ is a uniform limit of polynomials on $D$:
$$\varphi(t)=\lim_{m\rightarrow \infty} \text{diag }\{p_m(\lambda_1(t)),..., p_m(\lambda_n(t))\}.$$

Let $\mathcal{C}_n(\overline{\mathbb{R}})$ be the $C^*$-subalgebra of $(C(\overline{\mathbb{R}}))^n$ that consists of all 
 $n$-tuples $f=(f_1,...,f_n)$ such that
$$f_j(t)=f_k(x)$$
when  $\lambda_j(t)=\lambda_k(x)$.
We identify $f$ with $\text{diag}\, \{f_1,...,f_n\}$.

\begin{thm}
	The $C^*$-algebra $\mathcal{D}_n$ generated by $\phi^+$ is isomorphic to $\mathcal{C}_n(\overline{\mathbb{R}})$,
     where the  isomorphism is given by
	$$\varphi \longmapsto B^T \varphi B.$$
\end{thm}


\subsection{The algebra generated  by the Toeplitz Operators with symbols $a\in PC(\overline{\R}, \{0\})$.}\label{sep-est}

In this section, we describe the $C^*$-algebra generated by  all the Toeplitz operators $T_a$, or equivalently, the $C^*$-algebra generated by the
matrix-valued functions $\phi^a:\overline{\Pi}\rightarrow\C$, with $a \in PC(\overline{\R}, \{0\})$.
Let $\mathfrak{B}$ be the $C^*$-algebra generated
by all the matrix-valued functions  $\phi^a$, with $a\in PC(\overline{\R}, \{0\})$, and 
let $\mathcal{T}$ be the $C^*$-subalgebra  of $M_{n}(C(\overline{\Pi}))= M_{n}(\C) \otimes C(\overline{\Pi})$ 
consisting of all $M$ such that $M(\pm \infty, t_2) \in \C I$ for each $t_2\in\overline{\R}_+$ and
$$B^TM(\cdot, 0)B, \  B^TM(\cdot, +\infty)B \in \mathcal{C}_n(\overline{\mathbb{R}}).$$
We will prove that $\mathfrak{B}=\mathcal{T}$ by using a Stone-Weierstrass theorem for $C^*$-algebras.

\begin{thm}[\cite{Kaplansky}]\label{kaplansky}
	Let $\mathcal{A }$ and $\mathcal{B}$ be  $C^{*}$-algebras such that $\mathcal{B} \subset \mathcal{A}$. If $\mathcal{A}$ is a CCR algebra and $\mathcal{B}$  separates the pure state space of $\mathcal{A}$, then $\mathcal{B}=\mathcal{A}$.
\end{thm}

Our main result of this section is the following:

\begin{thm}
	The $C^*$-algebra generated by all  matrix-valued functions $\phi^a$, with $a \in PC(\overline{\R}, \{0\})$,  equals  $\mathcal{T}$. That is, the $C^*$-algebra generated by all Toeplitz operators  $T_a$ is isomorphic and isometric  to  the algebra $\mathcal{T}$, where the isomorphism is defined on the generators by the rule
	$$T_{a} \mapsto \phi^a.$$
\end{thm}

\begin{proof}
	$\mathfrak{B}=\mathcal{T}$ follows from Theorem \ref{kaplansky}. That is, $\mathfrak{B}$  separates the pure state space of $\mathcal{T}$ according to  Lemmas \ref{sep-lateral},
	\ref{separar-l-c}, \ref{separar-l-a}, \ref{separar-c-a}, \ref{separar-centro} and \ref{misma-fibra}.
\end{proof}

It is easy to see that $\mathfrak{B}$ is contained in $\mathcal{T}$.
Let $\langle \cdot , \cdot \rangle$ denote the usual inner product on $\C^n$.
Now the pure state space of the $C^*$-algebra $\mathcal{T}$ consists of 
all functionals having the form:
\begin{enumerate}
\item[1)]  $f_{(x_1,x_2), v}(M) = \langle M(x_1,x_2) v, v \rangle$ for $(x_1,x_2) \in \Pi$, $v \in \C^n$ a unit vector,
\item[2)]  $f_{(\pm \infty, t_2)}(M) = \lambda_{\pm t_2}$ for $0\leq t_2 \leq +\infty$, where $\lambda_{\pm t_2}I=M(\pm \infty, t_2)$,
\item[3)] $f_{(t_1,\pm \infty), j}(M)=\langle M(t_1,\pm \infty) v_j(t_1), v_j(t_1) \rangle$ for $t_1\in \overline{\R}$ and $j=,...,n$,
\end{enumerate}
where $M\in \mathcal{T}$ is arbitrary.
 
 \begin{lem}\label{sep-lateral}
 Let $t_2, \tau_2 \in \overline{\R}_+$. 
 We have
 $f_{(-\infty, t_2)}(\phi ^{\chi_+})\neq f_{(+ \infty, \tau_2)}(\phi^{\chi_+}).$
 If $t_2 \neq \tau_2$,
 then there exists 	$a \in C(\overline{\R})$ such that
 $f_{(\pm\infty, t_2)}(\phi ^a)\neq f_{(\pm \infty, \tau_2)}(\phi^a).$
 \end{lem}

\begin{proof}
	The pure states  $f_{(-\infty, t_2)}$ and $f_{(+ \infty, \tau_2)}$
	are separated by $\phi^{\chi_+}$ 
	since $\phi^{\chi_+}(-\infty, +\infty)= I$ and $\phi^{\chi_+}(+\infty, +\infty)= 0$. 
If
$a \in C(\overline{\R})$, then $\phi^{a}(\pm \infty, t_2)= a(\mp 1/(2\sqrt{t_2}))I$ for $t_2\in \mathbb{R}+$, 
$\phi^{a}(\pm \infty, 0)=a(\mp \infty)I $ and $\phi^{a}(\pm \infty, +\infty)=a(0)I$. Thus, taking
$a(s)= s/\sqrt{s^2 +1}$ we have 
$$f_{(\pm \infty, t_2)}(\phi^a)= \mp \frac{1}{\sqrt{1+4t_{2}}}.$$
Therefore, the pure states $f_{(\pm \infty, t_2)}$  and $f_{(\pm \infty, \tau_2)}$ are separated by $\phi^a$
when $t_2\neq \tau_2$. 
\end{proof}

We shall continue  separating the rest of pure states using continuous functions on $\overline{\R}$ and the indicator function $\chi_+$.

Let $v \in \C^n$ be a unit vector. Consider the function $h_{v}(s)=|\langle H(s), v \rangle|^2 $. This can be written as $h_{v}(s)=q_v (s)e^{-s^2}$, where
$$q_v(s)=|v_0 d_0H_0(s)  + \cdots + v_{n-1}d_{n-1}H_{n-1}(s)|^2$$
is a polynomial of degree at mosts $2n-2$ taking non-negative values.

\begin{lem}\label{separar-l-c}
	Let $v \in \C^n$ be a unit vector, and $(\pm\infty, t_2)$, $(x_1, x_2)$ be points with  
	 $x_1 \in \R $ and $x_2, t_2 \in \overline{\R}_+$. Then there exists a  symbol 
	$a \in PC(\overline{\R}, \{0\})$ such that
	$$f_{(\pm\infty, t_2)}(\phi ^a)\neq f_{(x_1, x_2),v}(\phi^a).$$
	
\end{lem}

\begin{proof}
	For $x_1 \in \mathbb{R}$ and $x_2  \in \overline{\R}_+$, we have
	$$f_{(x_1, x_2),v}(\phi^{\chi_+}) =\int_{x_1}^{\infty}   |\langle H(s), v \rangle|^2 ds
	=\int_{x_1}^{\infty} q_v(s)e^{-s^2} ds.  $$
	Since $q_v$ is not zero and non-negative,  $f_{(x_1, x_2),v}(\phi^{\chi_+}) >0$. 
	On the other hand, $f_{(+\infty, t_2)}(\phi ^{\chi_+})= \chi_+ (-1/(2\sqrt{t_2}))=0$ for $t_2\in [0,+\infty]$. 
	Hence
	$$f_{(+\infty, t_2)}(\phi ^{\chi_+}) \neq  f_{(x_1, x_2),v}(\phi^{\chi_+}).$$
	We now take $\chi_- =1-\chi_+$, then
	$$f_{(x_1, x_2),v}(\phi^{\chi_-}) =\int_{-\infty}^{x_1}   |\langle H(s), v \rangle|^2 ds
	=\int_{-\infty}^{x_1} q_v(s)e^{-s^2}>0.$$
	Also,   
	$f_{(-\infty, t_2)}(\phi ^{\chi_-})= \chi_- (1/(2\sqrt{t_2}))=0$.
	Then,
	$$f_{(-\infty, t_2)}(\phi ^{\chi_-}) \neq  f_{(x_1, x_2),v}(\phi^{\chi_-}).$$
\end{proof}

\begin{lem}\label{separar-l-a}
	Let $v,w \in \C^n$ be unit vectors, and let $(t_1, 0) ,(x_1, x_2) \in \overline{\Pi}$, with $t_1 \in \R$, $x_1 \in\R$ and 
	$0<x_2\leq +\infty$. Then there exists a symbol $a \in C(\overline{\R})$ such that 
	$$ f_{(t_1, 0),w}(\phi^{a}) \neq f_{(x_1, x_2),v}(\phi^{a}).$$
\end{lem}

\begin{proof}
	Take  $a(s)=1/(s^2 +1)$, so  $a(-\infty)=0=a(+\infty)$. Then, 
	$$ f_{(t_1, 0),w}(\phi^{a})= a(-\infty)	\int^{t_1}_{-\infty} |\langle H(s), w \rangle|^2 ds +
	 a(+\infty) \int_{t_1}^{\infty} |\langle H(s), w \rangle|^2 ds = 0.$$
	
	Since $a(s) >0$ for all $s \in \R$ and  $q_v(s)=0$ at most at finite number of values of $s$, we have that 
	 $a(s)q_v(s)e^{-s^2}>0$ almost everywhere. Then,
	$$	f_{(x_1, x_2),v}(\phi^{a})=
	\int_{-\infty}^{\infty} a\left(\frac{-x_{1}+s}{2\sqrt{x_{2}(x_{1}^{2} +1)}}\right) q_v(s)e^{-s^2}ds >0
	\quad \text{if } x_2\in \mathbb{R}.$$
	Moreover, $\phi^a(x_1, +\infty)= a(0)I=1\cdot I$. Thus,
	$f_{(x_1, +\infty),v}(\phi^{a}) =1. $
	We have proved that $ f_{(t_1, 0),w}(\phi^{a}) \neq f_{(x_1, x_2),v}(\phi^{a})$.
\end{proof}

\begin{lem}\label{separar-c-a}
	Let $v,w \in \C^n$  be unit vectors, $(x_1,x_2) \in \Pi$ and $(t_1,+\infty)\in\overline{\Pi}$ with 
	 $ t_1 \in \R$. Then there exists $a \in C(\overline{\R})$ such that
	$$	f_{(x_1, x_2),v}(\phi^{a}) \neq f_{(t_1, +\infty),w}(\phi^{a}).$$
\end{lem}

\begin{proof}
	Let $a(s)=|s|/(s^2 +1)$. So $a(0)=0$ and  $f_{(t_1, +\infty),w}(\phi^{a})=a(0) = 0$.  
	Since $q_v(s)=0$ at most at finite number of points, 
	$$f_{(x_1, x_2),v}(\phi^{a})=
	\int_{-\infty}^{\infty} a\left(\frac{-x_{1}+s}{2\sqrt{x_{2}(x_{1}^{2} +1)}}\right) q_v (s)e^{-s^2}ds >0.$$
	\end{proof}

Next we will separate the pure states associated to the points $(t_1, t_2) \in\Pi$ using continuous symbols indexed by
 $\alpha>0$ and $r \in \R$.
 We introduce 
 $$a^r_{\alpha}(y)=\frac{1}{\alpha} a([y-r]/\alpha),$$
 where
 $$a(y)= \left\{ 
 \begin{array}{lcl}
 	0       &  \text{if} & y \notin [-1 ,1]\\ \\  
 	1 + y &   \text{if} & y \in [-1,  0]\\ \\ 
 	1 - y &   \text{if} & y \in [0, 1].
 \end{array}
 \right.$$
Note that the family of functions $a_{\alpha}:=a_{\alpha}^{0}$ is an approximate identity in 
$L^1(\R, d\mu)$. Since
$h_j h_k \in L^1(\R)$, we have pointwise convergence in
$$\lim_{\alpha \rightarrow 0} (a_{\alpha} * h_j h_k)(y) =( h_j h_k)(y)$$
because $ h_j h_k $ is continuous. 

Since $\Phi: \Pi \rightarrow \Pi$ is a homeomorphism and
$\phi^{a}=\gamma^{a} \circ \Phi^{-1}$, we consider the matrix-valued function $\gamma^a$
in order to carry out the separation of pure states associated to the points in $\Pi$.

\begin{lem}
Let  $(x_1, x_2) \in \Pi$ and $a_{\frac{\alpha}{2\sqrt{x_2}}}^r$, with $\alpha>0$ and  $r \in \R$.
Then the matrix-valued function $\gamma^{a_{\frac{\alpha}{2\sqrt{x_2}}}^r}$ satisfies 

$$\lim_{\alpha \rightarrow 0} \gamma^{a_{\frac{\alpha}{2\sqrt{x_2}}}^r}(x_1,x_2) =
2\sqrt{x_2} H(x_{1}+2 \sqrt{x_{2}}r)[H(x_{1}+2 \sqrt{x_{2}}r)]^T.$$
\end{lem}

\begin{proof}
Take into account that $\{a_{\alpha}\}$ is an approximate identity and
$a_{\alpha}(y-x)=a_{\alpha}(x-y)$ for any $x, y \in  \R$.
Calculate the entries of  $\gamma^{a_{\frac{\alpha}{2\sqrt{x_2}}}^r}$:
\begin{eqnarray*}
	\gamma_{jk}^{a_{\frac{\alpha}{2\sqrt{x_2}}}^r}(x_1,x_2) &=& 
	\int_{-\infty}^{\infty} a_{\frac{\alpha}{2\sqrt{x_2}}}^r \left( \frac{-x_{1}+y}{2\sqrt{x_{2}}} \right) 
	(h_{j-1}h_{k-1})(y) dy\\
	&=&	\int_{-\infty}^{\infty} a_{\frac{\alpha}{2\sqrt{x_2}}}\left( \frac{-x_{1}+y}{2\sqrt{x_{2}}} - r\right) 
	(h_{j-1}h_{k-1})(y) dy\\
	&=& \frac{2\sqrt{x_2}}{\alpha} 	\int_{-\infty}^{\infty} 
	a\left( \frac{-(x_{1}+2 \sqrt{x_{2}}r) +y}{2 \sqrt{x_{2}}}\cdotp \frac{2 \sqrt{x_2}}{\alpha} \right) (h_{j-1}h_{k-1})(y) dy\\
	&=& 2\sqrt{x_2} \int_{-\infty}^{\infty} \frac{1}{\alpha} a\left( \frac{-(x_{1}+2 \sqrt{x_{2}}r) +y}{\alpha} \right)
	(h_{j-1}h_{k-1})(y) dy \\
	&=& 2\sqrt{x_2} \int_{-\infty}^{\infty} a_{\alpha}(y -(x_{1}+2 \sqrt{x_{2}}r) )  (h_{j-1}h_{k-1})(y) dy \\
	&=& 2\sqrt{x_2} \int_{-\infty}^{\infty} a_{\alpha}((x_{1}+2 \sqrt{x_{2}}r)  -y)  (h_{j-1}h_{k-1})(y) dy \\
	&=& 2\sqrt{x_2}  ((h_{j-1}h_{k-1})*a_{\alpha})(x_{1}+2 \sqrt{x_{2}}r) .
\end{eqnarray*}

Since $a_{\alpha}$  is an approximate identity, we have that
\begin{eqnarray*}
	\lim_{\alpha \rightarrow 0} \gamma_{jk}^{a_{\frac{\alpha}{2\sqrt{x_2}}}^r}(x_1,x_2) 
	&=&\lim_{\alpha \rightarrow 0} 2\sqrt{x_2}  ((h_{j-1}h_{k-1})*a_{\alpha})(x_{1}+2 \sqrt{x_{2}}r)\\
	&=& 2\sqrt{x_2} \lim_{\alpha \rightarrow 0} ((h_{j-1}h_{k-1})*a_{\alpha})(x_{1}+2 \sqrt{x_{2}}r)\\
	&=& 2\sqrt{x_2} (h_{j-1}h_{k-1})(x_{1}+2 \sqrt{x_{2}}r).
\end{eqnarray*}
\end{proof}

Take into account that
\begin{eqnarray*}
\lim_{\alpha \rightarrow 0} f_{(x_1,x_2),v}(\gamma^{a_{\frac{\alpha}{2\sqrt{x_2}}}^r}) 
&=& \lim_{\alpha \rightarrow 0} \langle  \gamma^{a_{\frac{\alpha}{2\sqrt{x_2}}}^r}(x_1,x_2) v, v  \rangle \\
&=&  2\sqrt{x_2} \langle H(x_{1}+2 \sqrt{x_{2}}r)[H(x_{1}+2 \sqrt{x_{2}}r)]^T v, v  \rangle\\
&=& 2\sqrt{x_2} |\langle H(x_{1}+2 \sqrt{x_{2}}r), v \rangle|^2 .
\end{eqnarray*}

For $v$ and $w$ unit vectors,   and  $(t_1, t_2)\neq (x_1, x_2)$ $\in \Pi$, 
the following result says that  
$$f_{(t_1, t_2),w}(\gamma^{a_{\frac{\alpha}{4\sqrt{x_2t_2}}}^{r}}(t_1, t_2))  \neq 
f_{(x_1, x_2),v}(\gamma^{a_{\frac{\alpha}{4\sqrt{x_2t_2}}}^{r}}(x_1, x_2))$$
for some  $\alpha >0$ and $r \in \R$. 

\begin{lem}\label{separar-centro}
Let  $v$, $w$  $\in \C^n$ be unit vectors, $(t_1, t_2)$, $(x_1, x_2)$ $\in \Pi$ and $a_{\frac{\alpha}{4\sqrt{x_2 t_2}}}^r$, with $\alpha >0$ and $r\in \R$. If
\begin{equation}\label{separar3}
	f_{(x_1,x_2),v}(\gamma^{a_{\frac{\alpha}{4\sqrt{x_2t_2}}}^r})  =
	f_{(t_1,t_2),w}(\gamma^{a_{\frac{\alpha}{4\sqrt{x_2t_2}}}^r}) \quad \forall \alpha >0, r \in \R,
\end{equation}	
then $(t_1, t_2)= (x_1, x_2)$.
\end{lem}

\begin{proof}	
It is easy to see that
$$\gamma_{jk}^{a_{\frac{\alpha}{4\sqrt{x_2t_2}}}^r}(x_1,x_2)
=2\sqrt{x_2}  ((h_{j-1}h_{k-1})*a_{\frac{\alpha}{2\sqrt{t_2}}})(x_{1}+2 \sqrt{x_{2}}r)$$ and
$$\gamma_{jk}^{a_{\frac{\alpha}{4\sqrt{x_2t_2}}}^r}(t_1,t_2)
=2\sqrt{t_2}  ((h_{j-1}h_{k-1})*a_{\frac{\alpha}{2\sqrt{x_2}}})(t_{1}+2 \sqrt{t_{2}}r).$$

Since equation (\ref{separar3}) holds for all $\alpha >0$, we can take the limit  in both sides of it
when $\alpha \rightarrow 0$; then  for all $r  \in \R$ we have
\begin{eqnarray*}
	\lim_{\alpha \rightarrow 0}f_{(x_1,x_2),v}(\gamma^{a_{\frac{\alpha}{4\sqrt{x_2t_2}}}^r})  &=&
	\lim_{\alpha \rightarrow 0}	f_{(t_1,t_2),w}(\gamma^{a_{\frac{\alpha}{4\sqrt{x_2t_2}}}^r}), \\
	2\sqrt{x_2}|\langle H(2\sqrt{x_2}r + x_1) , v \rangle|^2 &= & 2\sqrt{t_2}|\langle H(2\sqrt{t_2}r + t_1) , w \rangle|^2, \\
	\sqrt{x_2} e^{-(2\sqrt{x_2}r + x_1)^2} q_v (2\sqrt{x_2}r + x_1) & = & \sqrt{t_2} e^{-(2\sqrt{t_2}r + t_1)^2} q_w (2\sqrt{t_2}r + t_1),
\end{eqnarray*}	
where $q_v$ and $q_w$ are polynomials of degree at most $2n-2$. Thus there is a constant $C\in\R$ such that 
\begin{equation}\label{separa4}
	e^{4(x_2 -t_2)r^2 + 4(x_1\sqrt{x_2}-t_1\sqrt{t_2})r + x_1^2 -t_1^2} = C\quad \forall r \in \R.
\end{equation}
Therefore (\ref{separa4}) holds if and only if  $x_1=t_1$ and $x_2= t_2$. 
\end{proof}

For the separation of pure states attached to the same fiber we will use the following lemma.

\begin{lem}[\cite{A-C-R-N}]\label{H-indep}
Let $y_1, ..., y_n$ be real numbers different from each other. Then $\{H(y_1), ..., H(y_n)\} $ is a basis for
 $\C^n$.
\end{lem}

\begin{lem}\label{misma-fibra}
Let  $w, v \in \C^n$ unit vectors and  $(x_1, x_2) \in  \Pi$. Take the matrix-valued functions $\gamma^{a_{\frac{\alpha}{2\sqrt{x_2}}}^{r_1}}$ and
$\gamma^{a_{\frac{\alpha}{2\sqrt{x_2}}}^{r_2}}$, where $r_1, r_2 \in \R$ and $\alpha >0$. Suppose that
\begin{equation}\label{separar5}
	f_{(x_1,x_2),w}(\gamma^{a_{\frac{\alpha}{2\sqrt{x_2}}}^{r_1}}
	\gamma^{a_{\frac{\alpha}{2\sqrt{x_2}}}^{r_2}})  =
	f_{(x_1,x_2),v}(\gamma^{a_{\frac{\alpha}{2\sqrt{x_2}}}^{r_1}}
	\gamma^{a_{\frac{\alpha}{2\sqrt{x_2}}}^{r_2}} ) \quad \forall \alpha >0, r_1,r_2 \in \R.
\end{equation}
Then $ w= \lambda v$, where $\lambda$ is a uni-modular complex number; that is,   
$f_{(x_1,x_2), w} =f_{(x_1,x_2), v}$.
\end{lem}

\begin{proof}
Define $\beta_i = 2\sqrt{x_2}r _i + x_1$,  $i=1,2$. We have seen that
$$\gamma^{a_{\frac{\alpha}{2\sqrt{x_2}}}^{r}}(x_1,x_2)=
2\sqrt{x_2}  (HH^T*a_{\alpha})(x_{1}+2 \sqrt{x_{2}}r) .$$

Take the limit when $\alpha \rightarrow 0 $  in both sides of the equality  (\ref{separar5}):
\begin{eqnarray*}
	\lim_{\alpha \rightarrow 0}	f_{(x_1,x_2),w}(\gamma^{a_{\frac{\alpha}{2\sqrt{x_2}}}^{r_1}}
	\gamma^{a_{\frac{\alpha}{2\sqrt{x_2}}}^{r_2}}) &=&
	\lim_{\alpha \rightarrow 0}		f_{(x_1,x_2),v}(\gamma^{a_{\frac{\alpha}{2\sqrt{x_2}}}^{r_1}}
	\gamma^{a_{\frac{\alpha}{2\sqrt{x_2}}}^{r_2}}),\\
	4x_2\langle H(\beta_1)[H(\beta_1)]^T H(\beta_2)[H(\beta_2)]^T w, w \rangle 
	&= & 
	4x_2\langle H(\beta_1)[H(\beta_1)]^T H(\beta_2)[H(\beta_2)]^T v, v \rangle ,\\
	\overline{w}^T \left( H(\beta_1)[H(\beta_1)]^T  \right) \left( H(\beta_2)[H(\beta_2)]^T  \right) w
	& = & \overline{v}^T \left( H(\beta_1)[H(\beta_1)]^T  \right) \left( H(\beta_2)[H(\beta_2)]^T  \right) v.
\end{eqnarray*}

The real number $[H(\beta_1)]^T H(\beta_2)$ could be zero only for a finite number 
of values of $\beta_1$ and $\beta_2$. By continuity,
$$\overline{w}^T H(\beta_1) [H(\beta_2)]^T   w
=  \overline{v}^T  H(\beta_1) [H(\beta_2)]^T  v.$$
Without lost of generality we can assume that $x_1=0, x_2=(1/4)$;	then
\begin{eqnarray*}
	\overline{w}^T H(r_1)[H(r_2)]^T  w & = & \overline{v}^T  H(r_1)[H(r_2)]^T   v,\\
	\overline{\langle w, H(r_1) \rangle} \langle w, H(r_2) \rangle &=&
	\overline{\langle v, H(r_1) \rangle} \langle v, H(r_2) \rangle.
\end{eqnarray*}
This equality holds for all $r_1$ and $r_2$. In particular take $r=r_2=r_1$;
thus $|\langle w, H(r) \rangle|=|\langle v, H(r) \rangle|$ for all $r$.
We can write  $\langle w, H(r) \rangle=\langle v, H(r) \rangle e^{i\theta(r)}$ for all $r$.
Then
$$\overline{\langle v, H(r_1) \rangle} \langle v, H(r_2) \rangle  e^{i\theta(r_2)-i\theta(r_1)}
= \overline{\langle v, H(r_1) \rangle} \langle v, H(r_2) \rangle. $$

Thus  $e^{i\theta(r_2)-i\theta(r_1)}=1$ for all  $r_1, r_2$, which means that 
$\langle w, H(y) \rangle =e^{i\theta_0} \langle v, H(y) \rangle $ for all $y \in \R$ and some constant
$\theta_0$. Take $u=w- e^{i\theta_0}v$, then $\langle u, H(y) \rangle =0$.  
According to Lemma \ref{H-indep}, the set $\{H(y_k)\}_{k=1}^n $ is a basis for $\C^n$ and
$$\langle u, H(y_k) \rangle =0 , \quad k=1,...,n.$$
Therefore $u$ must be the zero vector.
\end{proof}

The $C^*$-algebra generated by all Toeplitz operators $T_c$ (with nilpotent symbol) is
large enough to fully describe its space of irreducible representations, for this reason 
we confine ourselves to consider a subclass of nilpotent symbols. 

\subsection*{Acknowledgments} The first author is grateful to CONAHCyT for its role in supporting this work through graduate and postdoctoral fellowships.  The authors also thank Matthew G. Dawson for  his useful comments and suggestions that improved this work. 

Some of the results in this paper were obtained in the research of the first author for her doctoral dissertation \cite{tesisYessica}. 


\end{document}